%% file: main.tex
\documentclass[11pt]{amsart}
\usepackage{amssymb,amsmath,amsthm,amsfonts,mathrsfs}
\usepackage[hmargin=3cm,vmargin=3.5cm]{geometry}
\usepackage[dvipsnames,table,xcdraw]{xcolor}
\usepackage[colorlinks = true,
            linkcolor = violet,
            urlcolor  = magenta,
            citecolor = OliveGreen,
            anchorcolor = blue]{hyperref}
\usepackage{graphicx,psfrag}
\usepackage{amscd}  
\usepackage[shellescape]{gmp}
\usepackage{stmaryrd}  %% double square brackets
\usepackage[all,2cell]{xy} \UseAllTwocells \SilentMatrices
\usepackage{tikz-cd}
\usepackage[dvips]{epsfig}
\usepackage{MnSymbol}
\usepackage{comment}
\usepackage{enumitem}
\usepackage{kbordermatrix}
\usepackage[colorinlistoftodos]{todonotes}

\catcode`~=11 % make LaTeX treat tilde (~) like a normal character
\newcommand{\urltilde}{\kern -.15em\lower .7ex\hbox{~}\kern .04em}  
\catcode`~=13 % revert back to treating tilde (~) as an active character

\renewcommand{\kbldelim}{(}% Left delimiter
\renewcommand{\kbrdelim}{)}% Right delimiter

\usetikzlibrary{arrows,automata}
\usetikzlibrary{decorations.markings}
\usetikzlibrary{decorations.pathmorphing}

\tikzset{snake it/.style={decorate, decoration=snake}}

\makeatletter
\def\@adminfootnotes{%
  \let\@makefnmark\relax  \let\@thefnmark\relax
  \ifx\@empty\@date\else \@footnotetext{\@setdate}\fi%%   <------ added
  \ifx\@empty\@subjclass\else \@footnotetext{\@setsubjclass}\fi
  \ifx\@empty\@keywords\else \@footnotetext{\@setkeywords}\fi
  \ifx\@empty\thankses\else \@footnotetext{%
    \def\par{\let\par\@par}\@setthanks}%
  \fi
}
\makeatother

\newtheorem{theorem}{Theorem}[section]
\newtheorem{proposition}[theorem]{Proposition}
\newtheorem{prop}[theorem]{Proposition}

\newtheorem{corollary}[theorem]{Corollary}
\newtheorem{lemma}[theorem]{Lemma}
\newtheorem{question}[theorem]{Question}

\theoremstyle{definition}
\newtheorem{definition}[theorem]{Definition}

\newtheorem{example}[theorem]{Example}

\theoremstyle{remark}
\newtheorem{remark}[theorem]{Remark}

\numberwithin{equation}{section}

\let\oldtocsection=\tocsection
\let\oldtocsubsection=\tocsubsection
\renewcommand{\tocsection}[2]{\hspace{0em}\oldtocsection{#1}{#2}}
\renewcommand{\tocsubsection}[2]{\hspace{1em}\oldtocsubsection{#1}{#2}}
\usepackage{chngcntr}
\counterwithin{figure}{subsection}

\makeatletter
\@namedef{subjclassname@2020}{%
  \textup{2020} Mathematics Subject Classification}

\makeatother
 
\title[Universal construction, Brauer envelope, pseudocharacters]{Universal construction in monoidal and non-monoidal settings, the Brauer envelope, and pseudocharacters}

\author{Mee Seong Im}
\address{Department of Mathematics, United States Naval Academy, Annapolis, MD 21402, USA}
\email{\href{mailto:meeseongim@gmail.com}{meeseongim@gmail.com}}
\thanks{}
 
\author{Mikhail Khovanov} 
\address{Department of Mathematics, Columbia University, New York, NY 10027, USA}
\email{\href{mailto:khovanov@math.columbia.edu}{khovanov@math.columbia.edu}}
\thanks{}

\author{Victor Ostrik} 
\address{Department of Mathematics, University of Oregon, Eugene, OR 97403, USA}
\email{\href{mailto:vostrik@math.uoregon.edu}{vostrik@math.uoregon.edu}}
\thanks{}

\subjclass[2020]{Primary: 57K16, 18M05, 18M30;
Secondary: 15A15}
\date{March 5, 2023}

\providecommand{\keywords}[1]{\textbf{\textit{Key words and phrases.}} #1}

\keywords{Topological theory, topological quantum field theory, TQFT, defects, universal construction,  pseudocharacter, pseudo-TQFT, Brauer envelope.}

\begin{document}

\def\AND{\mathsf{AND}}

\def\threeup{\raisebox{0.0ex}{\scalebox{1}[1.5]{\rotatebox[origin=c]{90}{$\subset$}}} \hspace{-0.19cm}\raisebox{0.03ex}{\scalebox{1}[1.30]{\rotatebox[origin=c]{90}{$-$}}}}

\def\threedown{\raisebox{0.0ex}{\scalebox{1}[1.5]{\rotatebox[origin=c]{-90}{$\subset$}}} \hspace{-0.16cm}\raisebox{0.06ex}{\scalebox{1}[1.30]{\rotatebox[origin=c]{90}{$-$}}}}

\def\cupbullet{\raisebox{0.0ex}{\scalebox{1}[1.50]{\rotatebox[origin=c]{0}{$\cup$}}}\hspace{-0.21cm}\raisebox{-0.45ex}{\scalebox{1}[1.00]{$\bullet$}} }

\def\capbullet{\raisebox{0.0ex}{\scalebox{1}[1.50]{\rotatebox[origin=c]{0}{$\cap$}}}\hspace{-0.21cm}\raisebox{1.10ex}{\scalebox{1}[1.00]{$\bullet$}} }

\def\tallcupleft{\raisebox{-0.15ex}{\scalebox{1}[2.75]{\rotatebox[origin=c]{180}{$\curvearrowright$}}}}

\def\tallcupright{\raisebox{-0.15ex}{\scalebox{1}[2.75]{\rotatebox[origin=c]{180}{$\curvearrowleft$}}}}

\def\tallcapleft{\raisebox{-0.35ex}{\scalebox{1}[2.75]{\rotatebox[origin=c]{0}{$\curvearrowleft$}}}}

\def\tallcapright{\raisebox{-0.35ex}{\scalebox{1}[2.75]{\rotatebox[origin=c]{0}{$\curvearrowright$}}}}

\def\cupleft{\raisebox{-0.25ex}{\scalebox{1}[1.75]{\rotatebox[origin=c]{180}{$\curvearrowright$}}}}

\def\cupright{\raisebox{-0.25ex}{\scalebox{1}[1.75]{\rotatebox[origin=c]{180}{$\curvearrowleft$}}}}

\def\capleft{\raisebox{0.78ex}{\scalebox{1}[-1.75]{\rotatebox[origin=c]{180}{$\curvearrowright$}}}}

\def\capright{\raisebox{0.78ex}{\scalebox{1}[-1.75]{\rotatebox[origin=c]{180}{$\curvearrowleft$}}}}

\def\Char{\mathsf{char}}
\def\concatenate{\mathsf{concatenate}}
\def\gen{\mathsf{generators}}
\def\GL{\mathsf{GL}}
\def\PGL{\mathsf{PGL}}
\def\init{\mathsf{in}}
\def\t{\mathsf{t}}
\def\out{\mathsf{out}}
\def\I{\mathsf I}
\def\R{\mathbb R}
\def\Q{\mathbb Q}
\def\Z{\mathbb Z}
\def\mc{\mathcal{c}}
\def\FL{\mathsf{F}\mathscr{L}}
\def\finite{\mathsf{finite}}
\def\infinite{\mathsf{infinite}}
\def\N{\mathbb N} 
\def\C{\mathbb C}
\def\SL{\mathsf{SL}}
\def\SS{\mathbb S} 
\def\CP{\mathbb P}
\def\Ob{\mathsf{Ob}}
\def\op{\mathsf{op}}
\def\new{\mathsf{new}}
\def\old{\mathsf{old}}
\def\OR{\mathsf{OR}}
\def\AND{\mathsf{AND}}
\def\rat{\mathsf{rat}}
\def\rec{\mathsf{rec}}
\def\Rep{\mathsf{Rep}}
\def\tail{\mathsf{tail}}
\def\coev{\mathsf{coev}}
\def\ev{\mathsf{ev}}
\def\id{\mathsf{id}}
\def\s{\mathsf{s}}
\def\t{\mathsf{t}}
\def\start{\textsf{starting}}
\def\Notation{\textsf{Notation}}
\renewcommand\SS{\ensuremath{\mathbb{S}}}
\newcommand{\kllS}{\kk\llangle  S \rrangle} %% power ser
\newcommand{\kllSS}[1]{\kk\llangle  #1 \rrangle}
\newcommand{\klS}{\kk\langle S\rangle}  % nc polynomials
\newcommand{\aver}{\mathsf{av}}  % average 
\newcommand{\ophana}{\overline{\phantom{a}}}
\newcommand{\Bool}{\mathbb{B}}
\newcommand{\dmod}{\mathsf{-mod}}
\newcommand{\dfgpmod}{\mathsf{-fgpmod}}
\newcommand{\lang}{\mathsf{lang}}
\newcommand{\pfmod}{\mathsf{-pfmod}}
\newcommand{\primitive}{\mathsf{irr}}
\newcommand{\Bmod}{\Bool\mathsf{-mod}}  % B-module 
\newcommand{\Bmodo}[1]{\Bool_{#1}\mathsf{-mod}}  
\newcommand{\Bfsmod}{\Bool\mathsf{-}\underline{\mathsf{fmod}}}  % stable category 
\newcommand{\undvar}{\underline{\varepsilon}} %sequence of varepsilons, not using anymore
\newcommand{\undotimes}{\underline{\otimes}}
\newcommand{\cl}{\mathsf{cl}}
\newcommand{\PP}{\mathcal{P}} % powerset 
\newcommand{\whA}{\widehat{A}}
\newcommand{\whC}{\widehat{C}}
\newcommand{\Sigmaa}{\Sigma^{\ast}}
\newcommand{\Sigmac}{\Sigma^{\circ}}
\newcommand{\mcI}{\mathcal{I}}
\newcommand{\mcB}{\mathcal{B}}
\newcommand{\mcF}{\mathcal{F}}
\newcommand{\psdeg}{\mathsf{deg}}

\newcommand{\alphai}{\alpha_I}  % alpha vertical
\newcommand{\alphac}{\alpha_{\circ}}  % alpha circle 
\newcommand{\alphap}{(\alphai,\alphac)} % alpha pair 
\newcommand{\alphalr}{\alpha_{\leftrightarrow}}
\newcommand{\alphaZ}{\alpha_{\Z}}
\newcommand{\mcCinfalpha}{\mcC^{\infty}_{\alpha}}
\newcommand{\mathT}{\mathsf{T}}
\newcommand{\mathF}{\mathsf{F}}
\newcommand{\mcS}{\mathcal{S}}
\newcommand{\GCcross}{\mcG_{\ell} \times_{\mcC}\mcG_r}

% redefine emptyset symbol 
\let\oldemptyset\emptyset
\let\emptyset\varnothing

\newcommand{\undempty}{\underline{\emptyset}}
\def\basis{\mathsf{basis}}
\def\irr{\mathsf{irr}} % recognizable series 
\def\spanning{\mathsf{spanning}}
\def\elmt{\mathsf{elmt}}

\def\l{\lbrace}
\def\r{\rbrace}
\def\o{\otimes}
\def\lra{\longrightarrow}
\def\Ext{\mathsf{Ext}}
\def\mf{\mathfrak} 
\def\mcC{\mathcal{C}}
\def\mcO{\mathcal{O}}
\def\Fr{\mathsf{Fr}}

\def\ovb{\overline{b}}
\def\tr{{\sf tr}} 
\def\det{{\sf det }} 
\def\tral{\tr_{\alpha}}
\def\one{\mathbf{1}}   % unit  object of category 

\def\lra{\longrightarrow}
\def\twoheadlra{\longrightarrow\hspace{-4.6mm}\longrightarrow}
\def\hooklra{\raisebox{.2ex}{$\subset$}\!\!\!\raisebox{-0.21ex}{$\longrightarrow$}}
\def\kk{\mathbf{k}}  %% base field  
\def\gdim{\mathsf{gdim}}  %% graded dimension 
\def\rk{\mathsf{rk}}
\def\undep{\underline{\epsilon}}
\def\mathM{\mathbf{M}}  % boolean matrix 

% cobordism categories 
\def\CCC{\mathcal{C}} % cat of cobordisms 
\def\wCCC{\widehat{\CCC}}  % completed category

\def\complement{\mathsf{comp}}
\def\Rec{\mathsf{Rec}} % recognizable series  

\def\Cob{\mathsf{Cob}} 
\def\Kar{\mathsf{Kar}}   % Karoubi envelope 

\def\dmod{\mathsf{-mod}}   % modules  
\def\pmod{\mathsf{-pmod}}    % projective modules 
\def\Z{\mathbb{Z}}

\newcommand{\brak}[1]{\ensuremath{\left\langle #1\right\rangle}}
\newcommand{\oplusop}[1]{{\mathop{\oplus}\limits_{#1}}}
\newcommand{\ang}[1]{\langle #1 \rangle } 
\newcommand{\ppartial}[1]{\frac{\partial}{\partial #1}} %partial derivative 
\newcommand{\checkr}{{\bf \color{red} CHECK IT}}
\newcommand{\checkb}{{\bf \color{blue} CHECK IT}}
\newcommand{\checkk}[1]{{\bf \color{red} #1}}

\newcommand{\mcA}{{\mathcal A}}
\newcommand{\mcG}{{\mathcal G}}
\newcommand{\cZ}{{\mathcal Z}}
\newcommand{\sq}{$\square$}
\newcommand{\bi}{\bar \imath}
\newcommand{\bj}{\bar \jmath}
\newcommand{\BcC}{\mathcal{B}(\mathcal{C})}
\newcommand{\LcC}{L(\mcC)}
\newcommand{\undX}{\underline{X}}
\newcommand{\Sets}{\mathsf{Sets}} % category of sets
\newcommand{\kkB}{\kk\langle h_B\rangle }  % handle subalgebra of Frob algebra B

\newcommand{\undn}{\underline{n}}
\newcommand{\undm}{\underline{m}}
\newcommand{\undzero}{\underline{0}}
\newcommand{\undone}{\underline{1}}
\newcommand{\undtwo}{\underline{2}}

\newcommand{\cob}{\mathsf{cob}} % cobordism 
\newcommand{\comp}{\mathsf{comp}} % complementary

\newcommand{\Aut}{\mathsf{Aut}}
\newcommand{\Hom}{\mathsf{Hom}}
\newcommand{\Idem}{\mathsf{Idem}}
\newcommand{\Ind}{\mbox{Ind}}
\newcommand{\Id}{\textsf{Id}}
\newcommand{\End}{\mathsf{End}}
\newcommand{\iHom}{\underline{\mathsf{Hom}}}
\newcommand{\Bools}{\Bool^{\mathfrak{s}}}
\newcommand{\mfs}{\mathfrak{s}}

\newcommand{\drawing}[1]{
\begin{center}{\psfig{figure=fig/#1}}\end{center}}

\def\endomCempt{\End_{\mcC}(\emptyset_{n-1})}

\def\MS#1{{\color{blue}[MS: #1]}}
\def\MK#1{{\color{red}[MK: #1]}}

\begin{abstract} This paper clarifies basic definitions in the universal construction of topological theories and monoidal categories. The definition of the universal construction is given for various types of monoidal categories, including rigid and symmetric. It is also explained how to set up the universal construction for non-monoidal categories. The second part of the paper explains how to associate a rigid symmetric monoidal category to a small category, a sort of the Brauer envelope of a category. The universal construction for the Brauer envelopes generalizes some earlier work of the first two authors on automata, power series and topological theories. Finally, the theory of pseudocharacters (or pseudo-representations), which is an essential tool in modern number theory, is interpreted via one-dimensional topological theories and TQFTs with defects. The notion of a pseudocharacter is studied for Brauer categories  and the lifting property to characters of semisimple representations is established in characteristic 0 for Brauer categories with at most countably many objects.  
The paper contains a brief discussion of pseudo-holonomies, which are functions from loops in a manifold to real numbers similar to traces of the holonomies along loops of a connection on a vector bundle on the manifold.
It concludes with a classification of pseudocharacters (pseudo-TQFTs) and their  generating functions for the category of oriented two-dimensional cobordisms in the characteristic 0 case.
\end{abstract}

\maketitle
\tableofcontents

%%%%%%%%%%%%%%%%%%%%%%%
%
%  Intro  
%
%%%%%%%%%%%%%%%%%%%%%%%

\section{Introduction}
\label{section:intro}

This paper clarifies the basic setup for the universal construction of topological theories and monoidal categories. 
Furthermore, it proposes a generalization of the Brauer category that starts with a category $\mcC$ and forms a category of one-dimensional cobordisms decorated by objects and morphisms of $\mcC$. The universal construction for the Brauer category of $\mcC$ is investigated. It is also explained that the notion of a pseudocharacter, essential in modern algebraic number theory, can be described in the language of the universal construction for the Brauer category of a category with one object, as a lifting property from a topological theory to a one-dimensional topological quantum field theory (TQFT) with defects.  

The universal construction~\cite{BHMV,FKNSWW,Kh1,RW1,Kh2,KS3,Me} starts with a category $\mcC$ of cobordisms or a similar monoidal category and a multiplicative evaluation $\alpha$ of closed cobordisms (more generally, a multiplicative evaluation of the abelian monoid $\End_{\mcC}(\one)$ of endomorphisms of the identity object). The evaluation takes values in some commutative ring or semiring $R$. One then passes to the category $R\mcC$ of $R$-linear combinations of morphisms in $\mcC$. Category $R\mcC$ carries an equivalence relation, where two morphisms are equivalent if and only if no matter how they are extended to endomorphisms of the unit object $\one$ and evaluated via $\alpha$, the evaluations are equal. The resulting quotient category $\mcC_{\alpha}$ is $R$-linear and monoidal and often gives rise to a topological theory, a weak analogue of a TQFT, which to an object $X$ of $\mcC$ associates its state space $A(X):=\Hom_{\mcC_{\alpha}}(\one,X)$ and to a morphism $f\in \Hom_{\mcC}(X_1,X_2)$ a map $A(f):A(X_1)\lra A(X_2)$. This assignment is a lax monoidal functor $A_{\alpha}$ from $\mcC$ to the category $R\dmod$ of $R$-modules: 
\begin{eqnarray}
   & &  A_{\alpha}:\mcC\lra R\dmod, \ \ 
   X\longmapsto A(X), \\
   & &  A(X_1)\otimes_R A(X_2) \lra A(X_1\otimes X_2), \ \  X_1, X_2\in\Ob(\mcC). \label{eq_lax}
\end{eqnarray}
It is lax monoidal in the sense that maps in \eqref{eq_lax} exist but are not isomorphisms, in general (isomorphism property would make the functor $A_{\alpha}$ monoidal). 

\vspace{0.07in} 

This paper consists of three parts: 
\begin{itemize}
\item Section~\ref{sec-ucm} explain the universal construction in general monoidal categories and non-monoidal categories.
\item Section~\ref{sec_brauer} describes the Brauer category of $\mcC$ of one-cobordisms between zero-manifolds decorated by objects and morphisms of $\mcC$.
It also explains its generalization when one-cobordism may have \emph{inner} endpoints that do not belong to the outer boundary of the cobordism. Universal construction for Brauer categories is then considered. 
\item Sections~\ref{sec-pseudo} and~\ref{sec_2D_pseudo} explore
 the relation between pseudocharacters and topological theories. 
\end{itemize}

In Section~\ref{subsec_general} we give a definition of the universal construction in general monoidal categories, requiring a multiplicative homomorphism $\alpha:\End_{\mcC}(1)\lra R$ from the abelian monoid of endomorphisms of the identity object to a commutative semiring $R$. We explain how various properties of $R$ and $\mcC$, including $R$ being a commutative ring and $\mcC$ being symmetric or rigid, simplify the definition and study of the universal construction. 
 
In Section~\ref{subsec_nonmon} we propose a setup for the universal construction when category $\mcC$ is not monoidal. One then substitutes the unit object $\one$ by collections of starting and ending objects. The state space of an object $X$ is an $R$-module generated by morphisms from starting objects to $X$ subject to relations on $R$-linear combinations on these morphisms that come from composing with morphisms from $X$ to ending objects and evaluating compositions via $\alpha$. We further generalize this setup in Section~\ref{subsec_presheaves}, replacing collections of starting and ending objects by a pair of presheaves of sets, one for $\mcC$ and one for $\mcC^{\op}$.

\vspace{0.07in} 

Section~\ref{subsec_brauer} defines \emph{the Brauer envelope} or \emph{the Brauer category} $\BcC$ of a small category $\mcC$. One can think of it as the category of oriented one-dimensional cobordisms decorated by morphisms of $\mcC$ between oriented zero-manifolds decorated by objects of $\mcC$. Composition of basic cobordisms produces loops decorated by endomorphisms in $\mcC$ subject to a suitable equivalence relation. That section also explains the Poincar\'e dual construction, where cobordisms carry $0$-dimensional defects decorated by morphisms of $\mcC$ while intervals bounded by these defects are decorated by objects of $\mcC$. 

Section~\ref{subsec_uc_Brauer} describes the universal construction for the Brauer category, where each loop $x$ in $\BcC$ is assigned an element $\alpha(x)$ in a commutative semiring $R$. This results in an $R$-linear rigid symmetric monoidal category $\mcB_{\alpha}(\mcC)$ and a chain of categories and functors 
\[
\mcC \lra \BcC \lra R\BcC \lra \mcB_{\alpha}(\mcC),
\]
with the first arrow -- an ``inclusion'' of $\mcC$ into its Brauer envelope, the second arrow being $R$-linearization of the Brauer category and the third arrow given by the universal construction for evaluation $\alpha$. 

In Section~\ref{subsec_Brauer_boundary} we consider more general $\mcC$-decorated one-cobordisms that may terminate (end) in the ``middle'' of the cobordism, that is, not at the boundary of the cobordism which describes the source and target of a morphism in $\BcC$. These additional endpoints are called \emph{inner} or \emph{floating} endpoints. Section~\ref{subsec_presheaves}  describes a general framework for such cobordisms, via a pair of presheaves of sets on $\mcC$ and $\mcC^{\op}$, and explains the universal construction in this case. Non-monoidal universal construction for this setup is exhibited there as well. 

\vspace{0.07in} 

Section~\ref{sec-pseudo} starts by explaining the lifting problem of realizing evaluations via TQFTs, see Section~\ref{subsec_realization}. 
An essential tool in the modern theory of Galois representations is the theory of 
pseudocharacters (or pseudorepresentations)~\cite{Taylor91,Carayol94,Rou-pseudo96,Nyssen96,Dot11,Bella12,Bella21,WW17,Wang13thesis}. Pseudocharacters can be traced back to the turn of the 20th century, when related structures (group determinants) motivated Frobenius to develop the theory of characters of finite groups, see~\cite{Johnson19} and~\cite[Remark 1]{Dot11}. 
We explain that pseudocharacters can be thought of as evaluations $\alpha$ for the Brauer category of a category $\mcC_G$ with a single object $X$ and the group or monoid $G$ of endomorphisms of $X$. Evaluations that are pseudocharacters have an additional property that the antisymmetrizer of $X^{\otimes (d+1)}$ is identically $0$ for some $d$. 
Surprisingly, that is often enough to imply that the evaluation can be realized as the character of a representation. 
A realization of a pseudocharacter as the character of a representation $V$ can be interpreted as the lifting of the corresponding evaluation (or a topological theory) to a one-dimensional TQFT with defects. This can be presented as a diagram in Figure~\ref{figure-A0}. 

\vspace{0.07in} 

\input{figure-A0}

In Section~\ref{subsec_distributed} we consider the distributed case of this correspondence, when $\mcC$ has more than one object. One potential application of the distributed case is to \emph{pseudo-holonomies}, that is, functions on closed paths in a manifold $M$ that behave like traces of the holonomy of a connection on a vector bundle over $M$, see Section~\ref{subsec_holonomies}. We also discuss the case of inner endpoints, so that an evaluation is defined for both decorated circles and intervals, see Section~\ref{subsec_pseudo_boundary}. 

\vspace{0.07in}

In Section~\ref{sec_2D_pseudo} we study pseudo-characters beyond the one-dimensional case. We review two-dimensional TQFTs and their generating functions. A pseudo-character $\alpha$ for the category $\Cob_2$ of two-dimensional cobordisms produces a pseudo-character for the Brauer category with inner endpoints and one unlabelled defect. Additional properties of that pseudo-character (and restricting to fields of characteristic $0$) allow to constraint the generating function of $\alpha$ and show that it coincides with the generating function of some two-dimensional TQFT. The resulting classification of pseudocharacters of $\Cob_2$ is stated in Theorem~\ref{thm_classif}.   

Pseudocharacters for more general monoidal categories may find applications beyond number theory, for instance, in the deformation theory for TQFTs. 

\vspace{0.07in} 

Calling the rigid tensor envelope $\BcC$  the Brauer category of $\mcC$ is motivated by the following example. The familiar oriented Brauer category, see~\cite{BCNR17} for an exposition, is a special case of our construction when $\mcC$ is a category with one object $X$ and a single morphism (the identity morphism of $X$). 
Then endomorphisms of the identity object in the rigid symmetric monoidal category $\BcC$ are powers of the circle $\SS_X$ (a circle decorated by the morphism $\id_X$), see Figure~\ref{figure-D0}. 
One can then pass to the $R$-linear extension $R \BcC$ of $\mcC$ by forming linear combinations of morphisms, and then impose the relation
 $\alpha(\SS_X)=\lambda\in R$ saying that a circle evaluates to $\lambda$, for a commutative ring $R$, often a field. The resulting category is  usually called the oriented Brauer category. 

 The next step is to pass to the negligible quotient of this category, which is equivalent to doing the universal construction for this evaluation $\alpha$. 
When $|\lambda|=n\in \Z_+$, the negligible quotient can be interpreted via the representation category of $\GL(n)$.  

Deligne and Milne constructed a family of rigid, symmetric monoidal categories $\Rep(\GL(\lambda))$, where $\lambda\in\mathbb{C}$, see \cite{DM82,Deligne90}. When $\lambda\not\in\Z$, $\Rep(\GL(\lambda))$ is a semisimple tensor category satisfying a certain universal property. Further developments of the semisimplification of the tilting categories of $\GL(n)$, $\SL(n)$, and $\PGL(n)$ in prime characteristic appear in~\cite{BEEO20,EO22}. 
 
\vspace{0.07in} 

More generally, when $\mcC$ has a single object $X$, the category $\mcC$ can be described by the monoid $G$ of endomorphisms of $X$ and denoted $\mcC_G$. It is this case that relates to pseudocharacters.  
Beyond the topological theory interpretation of pseudocharacters, Brauer categories $\mcB(\mcC_G)$,  together with the interpolation categories $\mcB_{\alpha}(\mcC_G)$, naturally appear in the topological theory and TQFT interpretations of formal rational power series, regular languages and automata, see~\cite{Kh3,IK-top-automata,IK,GIKKL23}.

Brauer categories $\mcB(\mcC_G)$ are an intermediate step in the construction of Frobenius Heisenberg categories in \cite{BSW21} (for that relation one needs to pick a set of generators of the Frobenius algebra, to label the dots on a line). 

In the more restricted case, 
when $\mcC$ is free on a single object $X$ and a generating morphism $x:X\lra X$ (besides the identity morphism), variations on the Brauer category $\BcC$ 
appear as  intermediate categories in various  categorifications of the Heisenberg algebra~\cite{Kho14,BSW20}. Replacing  the symmetric group by the nilHecke algebra, related categories appear in a categorification of the quantum $\mathfrak{sl}(2)$, see~\cite{Lauda10}. 

Brauer categories with inner endpoints (where one allows one-manifolds to end in the middle of a cobordism), considered in Sections~\ref{subsec_Brauer_boundary}, \ref{subsec_presheaves}, and their interpolations generalize the rook Brauer algebra and category~\cite{HdelM,Hu20}, which correspond to the case of $\mcC$ with a single object, single (identity) morphism and its action on one-element sets. Inner endpoints describing morphisms to and from the unit object in monoidal categories also appear, for instance, in~\cite{KS1,Kh3,IK22-linearcase,KT2019}. 

%Paper~\cite{Coulembier21} deals with a more subtle question of starting a monoidal category and adding duals for some objects. 

\vspace{0.1in} 

{\bf Acknowledgments:} M.K. and M.S.I. would like to thank Dublin Institute for Advanced Studies and Sergei Gukov for providing productive working environment during a workshop in November 2022 where this work was started. The authors are grateful to Yakov Kononov and Eric Urban for interesting discussions. 
 M.K. would like to acknowledge partial support from NSF grant DMS-2204033 and Simons Collaboration Award 994328. 

%%%%%%%%%%%%%%%%%
%
% UC settings
%
%%%%%%%%%%%%%%%%%

\section{Universal construction in monoidal and non-monoidal settings}
\label{sec-ucm}

Throughout the paper it is assumed that all categories are small, including categories $\mcC$ below. 

%%%%%%%%%%%%%%%%
% general UC 
%%%%%%%%%%%%%%%%

\subsection{Universal construction for general monoidal categories}
\label{subsec_general} 

In the universal construction of topological theories, one starts with a monoidal category $\mcC$.  Endomorphisms $\End_{\mcC}(\mathbf{1})$ of the identity object constitute a commutative monoid. One picks a commutative ring or a commutative semiring $R$ and selects a multiplicative map 
\begin{equation}
    \alpha \ : \ \End_{\mcC}(\mathbf{1})\lra R, 
\end{equation}
so that 
\begin{equation}
    \alpha(1) = 1, \hspace{0.50cm}  
    \alpha(xy)=\alpha(x)\alpha(y), \hspace{0.50cm} 
    x,y\in \End_{\mcC}(\mathbf{1}). 
\end{equation}
The category $\mcC$ typically has a set-theoretic or topological origin, and hom spaces $\Hom_{\mcC}(X_1,X_2)$ in it are just sets. One passes to the $R$-linear closure $R\mcC$ of $\mcC$, which is a monoidal category with the same objects as $\mcC$, but morphisms in $R\mcC$ are finite $R$-linear combinations of morphisms in $\mcC$, with the composition given by extending composition in $\mcC$ bilinearly to linear combinations of morphisms. Evaluation $\alpha$ extends $R$-linearly to a homomorphism of commutative semirings, also denoted $\alpha$:  
\begin{equation}
    \alpha \ : \ \End_{R\mcC}(\mathbf{1})\cong R\End_{\mcC}(\mathbf{1})\lra R. 
\end{equation}
Next, define the category $\mcC_{\alpha}$ to be a  quotient category of $R\mcC$, with the same objects as in $R\mcC$ and $\mcC$. 
Two morphisms $f_1,f_2: X_1 \lra X_1' $ in $R\mcC$ are $\alpha$-equivalent if for any objects $X_0, X_2$ and any morphisms  
\begin{equation}
%\label{no_label_1}
%& & g:  X_0\lra X_0', \ \  \ h :  X_2\lra X_2', \\
\label{no_label_2}
%& & 
\threeup  :  \one \lra X_0\otimes X_1 \otimes X_2 , \ \ \ \ 
\threedown : X_0\otimes X_1'\otimes X_2 \lra \one
\end{equation}
in $\mcC$ (or, equivalently, in $R\mcC$)
the following relation holds 
\begin{equation}\label{eq_alpha_0}
  %  \alpha(\cap \circ (g\otimes f_1 \otimes h) \circ \cup) \ = \ \alpha(\cap \circ (g\otimes f_2 \otimes h) \circ \cup),
  \alpha(\threedown \circ (\id_{X_0}\otimes f_1 \otimes \id_{X_2}) \circ \threeup) \ = \ \alpha(\threedown \circ (\id_{X_0}\otimes f_2 \otimes \id_{X_2}) \circ \threeup),
\end{equation}
see Figure~\ref{fig-001-boundary}. 

\vspace{0.1in}

\input{fig-001-boundary}

\vspace{0.1in} 

In other words, we close up or complete $f_1,f_2$ in the same way to endomorphisms of $\one$ and require that the two closures evaluate to the same element of $R$.  

The category $\mcC_{\alpha}$ is  $R$-linear  monoidal. By an $R$-linear category, we mean a category where hom spaces are $R$-modules and composition of morphisms is $R$-bilinear. 

By the \emph{right state space} $A_{r}(X)$ of $X\in \Ob(\mcC)=\Ob(\mcC_{\alpha})$ we mean 
$\Hom_{\mcC_{\alpha}}(\one,X)$. 
Likewise, define the \emph{left state space} $A_{\ell}(X)$ of $X\in \Ob(\mcC)=\Ob(\mcC_{\alpha})$ as  
$\Hom_{\mcC_{\alpha}}(X,\one)$, so that  
\begin{equation}
A_{\ell}(X) := 
    \Hom_{\mcC_{\alpha}}(X,\one), \hspace{1cm} 
    A_{r}(X) := \Hom_{\mcC_{\alpha}}(\one,X). 
\end{equation}
Our convention for left vs. right state spaces comes from observing that the  composition 
\begin{equation}
\one \stackrel{g}{\longleftarrow}X \stackrel{f}{\longleftarrow} \one
\end{equation}
is written as $gf$, with $g\in \Hom(X, \one)$ on the left of $gf$ and $f\in \Hom(\one, X)$ on the right. 

Any morphism $f:X\lra X'$ in $\mcC$, in $R\mcC$ or in $\mcC_{\alpha}$, induces $R$-linear maps of state spaces 
\begin{equation}
    A_{r}(f): A_{r}(X)\lra A_{r}(X'),  
    \hspace{0.5cm} 
    A_{\ell}(f) : A_{\ell}(X')\lra A_{\ell}(X). 
\end{equation}
Maps $A_{r}(f)$, respectively $A_{\ell}(f)$, are covariant, respectively contravariant, for the composition of morphisms, $A_r(fg) = A_r(f)\circ A_r(g) $,  $A_{\ell}(fg)=A_{\ell}(g)\circ A_{\ell}(f)$.

The collection of state spaces $\{A_{r}(X)\}_{X\in \Ob(\mcC)}$ and maps $A_{r}(f)$, for $f\in \Hom_{\mcC}(X_1,X_2)$, $X_1,X_2\in \Ob(\mcC)$ gives a functor $A_{\alpha,r}$ (or just $A_r$) from $\mcC$ to $R\dmod$, which factors to a functor $\mcC_{\alpha}\lra R\dmod$, also denoted $A_r$. 
This functor 
\[A_r:\mcC\lra R\dmod
\]
describes a topological theory for $\mcC$, a weaker (or lax) version of a TQFT with 
morphisms
\[  
A_{r}(X_1) \otimes_R A_{r}(X_2) \lra  A_{r}(X_1\otimes X_2)
\] 
replacing corresponding isomorphisms in the definition of a TQFT, where $\otimes$ on the right stands for the tensor product in $\mcC$ and $\otimes_R$ on the left denotes the tensor product of $R$-semimodules over a commutative semiring $R$. Likewise, state spaces $A_{\ell}(X)$ and maps $A_{\ell}(f)$ for morphisms $f$ in $\mcC$, defined similarly to $A_r(f)$,  give a contravariant functor $A_{\alpha,\ell}$ (or just $A_{\ell}$) from $\mcC$ to $R\dmod$, which descends to a functor 
\[A_{\ell}:\mcC_{\alpha}\lra (R\dmod)^{\op},
\]
see Figure~\ref{figure-P6} for both functors $A_r,A_{\ell}$. 

\vspace{0.07in}

Since $R$ is commutative, there is no difference between left and right $R$-modules. Section~\ref{subsec_nonmon} discusses the non-monoidal version of the universal construction, for possibly non-commutative $R$, where the target categories for functors $A_{\ell}$ and $A_{r}$ differ by being those of left, respectively right $R$-modules, see Figure~\ref{figure-P6-2}. 

%{\bf Incorrect, remove:} Note the discrepancy in our terminology: \emph{right} state spaces $A_r(X)$ are \emph{left} $R$-modules, while \emph{left} state spaces  $A_{\ell}(X)$ are \emph{right} $R$-modules. The motivation for this terminology comes from calling the state space on the left, respectively right of the pairing \eqref{eq_pairing} the left, respectively right, state space of $X$. 

\vspace{0.07in} 

\input{figure-P6}

The pairing 
\begin{equation}\label{eq_pairing}
    (\:\:, \:\:)_X \ : \ A_{\ell}(X)\times A_{r}(X) \lra \End_{\mcC_{\alpha}}(\one) \cong R
\end{equation}
is $R$-bilinear (in the sense of linearity over a commutative semiring $R$) and nondegenerate, in the sense that for different elements $f,f'\in A_{\ell}(X)$ there exists $g\in A_{r}(X)$ with $(f,g)\not= (f',g)$, and likewise for the other coordinate.

\vspace{0.07in} 

\begin{remark} Instead starting with $\mcC$ and then first passing to $R\mcC$, one can start with a monoidal category $\mcC$ which is already $R$-linear and do the quotient construction for it given an $R$-linear evaluation $\alpha$, see also~\cite{Me}. 
\end{remark}

\vspace{0.07in} 

There are important special cases of this construction:

\vspace{0.06in} 

{\bf I.} {\it $R$ is a commutative ring rather than just a semiring.} In this case the terms can be moved to one side of the equation \eqref{eq_alpha_0}. Define $f:X_1\lra X_1'$ to be $\alpha$-equivalent to $0\in \Hom_{R\mcC}(X_1,X_1')$ if 
\begin{equation}\label{eq_alpha_0_2}
    \alpha(\threedown \circ (\id_{X_0}\otimes f \otimes \id_{X_2}) \circ \threeup) \ = 0
\end{equation}
for all objects and morphisms as in  \eqref{no_label_2}.  The collection $\mc{I}$ of morphisms $\alpha$-equivalent to zero morphisms between various pairs of objects is a two-sided ideal in $R\mcC$, and $\mcC_{\alpha}$ is isomorphic to the quotient category $R\mcC/\mc{I}$.

State spaces $A_{\ell}(X), A_r(X)$ are $R$-modules, and the nondegenerate pairing \eqref{eq_pairing} is that of $R$-modules. 

\vspace{0.07in} 

One can pass to the Karoubi closure $\Kar(\mcC_{\alpha})$ by allowing finite direct sums of objects and adding objects for idempotent morphisms. This is especially useful when $R$ is a commutative ring, or, even better, a field, see~\cite{KS3,KKO,KL,Me}. It is not clear how useful passing to the Karoubi closure is when $R$ is only a semiring, not a ring.   

\vspace{0.1in} 

{\bf II.} {\it $\mcC$ is symmetric monoidal.}  
In most papers on the universal construction, the category $\mcC$ is symmetric. Categories of cobordisms tend to be symmetric monoidal, while, for instance, the category of $1$-dimensional cobordisms embedded in $\R^2$ is monoidal but not symmetric, and the universal construction for such categories is investigated in~\cite{KL}.

With $\mcC$ symmetric, objects $X_0,X_2,X_0',X_2'$ can be flipped to one side of each tensor product, so that, for instance, one can reduce to $X_2 = X_2'=\one$. Then maps and the equation in \eqref{no_label_2}-\eqref{eq_alpha_0} simplify to 
\begin{eqnarray}\label{no_label_3}
& & 
\cupbullet : \one \lra X_0\otimes X_1 , \hspace{1cm} 
\capbullet : X_0\otimes X_1'\lra \one,    \\
\label{no_label_4}
& &  \alpha(\capbullet \circ (\id_{X_0}\otimes f_1 ) \circ \cupbullet) \ = \ \alpha(\capbullet \circ (\id_{X_0}\otimes f_2) \circ \cupbullet),
\end{eqnarray}
see Figure~\ref{fig-002-boundary}.
Category $\mcC_{\alpha}$ is then symmetric monoidal as well.

\vspace{0.1in} 

\input{fig-002-boundary}

\vspace{0.1in} 

{\bf III.}  {\it $\mcC$ is rigid monoidal.} 
Category $\mcC_{\alpha}$ is rigid~\cite{EGNO} if $\mcC$ is rigid. In this case morphisms and diagrams in \eqref{no_label_2}, \eqref{eq_alpha_0}  and  Figure~\ref{fig-001-boundary} admit only a slight simplification: the line of $X_0$ in Figure~\ref{fig-001-boundary} can be bend to create a cup and a cap so as to convert $\threedown$ to a morphism $X_1'\otimes X_2\lra X_0^{\ast}$ and convert $\threeup$ to a morphism $X_0^{\ast} \lra X_1\otimes X_2$. 

%can be simplified as follows: for objects $X_0,X_2$ and any  morphism  $g:X_1^{\ast}\otimes X_0^{\ast}\lra X_1^{\ast}\otimes X_0^{\ast}$ the two endomorphisms  of $\one$ obtained by closing up $\id_{X_0},f,g$ must have equal $\alpha$-evaluations, see Figure~\ref{fig-003-boundary}.

\vspace{0.1in} 

%\input{fig-003-boundary}

%\vspace{0.1in} 

Rigidity gives isomorphisms of left and right state spaces $A_{\ell}(X)\cong A_{r}(X)$ that are  intertwined by maps $f_{\ell}$ and $(f^{\ast})_r$, for $f\in \Hom_{\mcC}(X_1,X_2)$, where $f^{\ast}:X_2^{\ast}\lra X_1^{\ast}$ is the dual to $f$ morphism coming from the rigid structure. 

\vspace{0.1in}

{\bf IV.}  {\it $\mcC$ is rigid symmetric monoidal.} 
Then the category $\mcC_{\alpha}$ is rigid symmetric monoidal as well. This is the case when $\mcC$ is one of the categories of cobordisms, including categories of $n$-dimensional cobordisms,  cobordisms with defects, graphs or suitable CW-complexes viewed as cobordisms. The universal construction has mostly been studied in this case. 
Condition \eqref{no_label_2} can be further simplified in this case, requiring instead that for any morphism $g:X_1^{\ast}\lra (X_1')^{\ast}$ the two closures have the same $\alpha$-evaluation: 
\begin{eqnarray}\label{no_label_5}
& & 
\tallcupleft_{X_1} : \one \lra X_1\otimes X_1^{\ast} , \ \ 
\tallcapright_{X_1'} : X_1'\otimes (X_1')^{\ast}\lra \one    \\
\label{no_label_6}
& &  \alpha \left(
\tallcapright_{X_1'} \circ (f_1\otimes g ) \circ \tallcupleft_{X_1}\right) \ = \ \alpha\left(\tallcapright_{X_1'} \circ (f_2\otimes g) \circ \tallcupleft_{X_1} \right),
\end{eqnarray}
where $\tallcupleft_X,\tallcapright_X$ are the rigidity morphisms for the object $X$, also see Figure~\ref{fig-003-boundary}.

\vspace{0.1in} 

\input{fig-003-boundary}

\vspace{0.1in} 

If $R$ is, additionally, a ring, see Case {\bf I} above, one can further reduce to a single morphism 
$f:X_1\lra X_1'$ in $R\mcC$ and define for it to be $\alpha$-equivalent to the zero morphism if 
\begin{equation}\label{eq_equiv_rsm} 
     \alpha\left(\tallcapright_{X_1'} \circ (f\otimes g ) \circ \tallcupleft_{X_1}\right) \ = \ 0
\end{equation} 
for any morphism $g:X_1^{\ast}\lra (X_1')^{\ast}$ in $\mcC$. 
After that, one can form the quotient of $R\mcC$ by the ideal $I$ of morphisms $\alpha$-equivalent to zero morphisms. The quotient $R\mcC/I$ is an $R$-linear rigid symmetric monoidal category. 

\vspace{0.07in} 

The three properties: $R$ is a (commutative) ring, $\mcC$ is rigid, $\mcC$ is symmetric monoidal can be imposed independently, resulting in eight cases of the universal construction, with corresponding simplifications in the definition of the quotient category $\mcC_{\alpha}$. Five of these eight cases are considered in {\bf I}-{\bf IV} above. 

\vspace{0.07in} 

Papers on the universal construction for the most part specialize  to categories of $n$-dimensional cobordisms, often for manifolds with various decorations or for CW-complexes with decorations and $n=2$ (foams), see~\cite{BHMV,FKNSWW,Kh2,KS3,KL,KKO} and references therein.
Categories of cobordisms are usually rigid symmetric monoidal. 

Ehud Meir~\cite{Me} aptly renamed the universal construction to \emph{interpolation of monoidal categories}. Meir emphasized that the universal construction can be viewed in the general framework of rigid symmetric monoidal categories. In fact,  conditions that $\mcC$ be rigid and  symmetric can be dropped and one can consider interpolations for any monoidal category, as  explained above.  

Monoidal categories can, informally, be viewed as categories of generalized  one-dimensional cobordisms, especially when a monoidal category is given via generating objects, generating morphisms, and defining relations. Generating morphisms can be depicted as vertices with the ``in''  and ``out'' legs labelled by  generating objects. Such a morphism goes from the tensor product of objects for the ``in'' legs to the tensor product of objects for the ``out'' legs. Monoidal compositions of these morphisms are then decorated directed graphs, a kind of one-dimensional cobordisms.

%%%%%%%%%%%%%%%%%
% universal non-monoidal 
%%%%%%%%%%%%%%%%%

\subsection{Universal construction in a  non-monoidal setting} 
\label{subsec_nonmon} 

 It is possible to define a version of the universal construction in general categories, not necessarily monoidal. Suppose given a category $\mcC$ and a semiring $R$, not necessarily commutative.  
 Pick sets of objects $S_i$, $i\in I$ and $T_j, j\in J$. We refer to $S_i$'s as \emph{source} objects and to $T_j$'s as \emph{target} objects. For $i\in I,j\in J$, let $\alpha_{ji}$ be a map of sets 
\begin{equation}\label{eq_alphaji}
    \alpha_{ji} : \Hom_{\mcC}(S_i,T_j)\lra R. 
\end{equation}
Denote the union of these maps by $\alpha=(\alpha_{ji})_{i,j}$.

To define the state spaces in this setup, start with the free right $R$-module $\Fr_{r}(X)$ with a basis $\{[f]\}_f$, where $f\in \sqcup_{i\in I}\Hom_{\mcC}(S_i,X)$. 
That is, the basis elements are given by all morphisms from various starting objects $S_i$ to $X$, and an element of $\Fr_r(X)$ is a formal finite linear combination $\sum_k f_k a_k$, $a_k\in R$. Also define the free left $R$-module $\Fr_{\ell}(X)$ with a basis $\{[g]\}_g$, where $g\in \sqcup_{j\in J}\Hom_{\mcC}(X,T_j)$. Its elements are finite linear combinations $\sum_m b_m g_m$, $b_m\in R$. 

Define a pairing 
\begin{equation}
    (\:\:,\:\:)_X \ : \ \Fr_{\ell}(X)\times \Fr_r(X) \lra R 
\end{equation}
by 
\begin{equation}
    \left( [g],[f] \right)_X \ := \ \alpha_{ji}(gf), \ \ f\in \Hom(S_i,X), \ g\in \Hom(X,T_j) 
\end{equation}
and extending bilinearly. For finite linear combinations of $f_k$'s and $g_m$'s with $i$ and $j$ fixed, define  
\begin{equation}
    \left( \sum_m b_m[g_m],\sum_k [f_k] a_k \right)_X \ := \ \sum_{m,k} b_m \alpha_{ji}(g_m f_k) a_k , \ \ f_k\in \Hom(S_i,X), \ g_m\in \Hom(X,T_j) ,
\end{equation}
and then extend by linearity to finite linear combinations of terms with different indices $i$ and different indices $j$. Notice that $b_m, a_k$ and values of evaluations $\alpha_{ji}$ are elements of a semiring $R$, noncommutative in general, and the order of the elements in the product is important. 

Two elements $g_1,g_2\in \Fr_{\ell}(X)$ are called \emph{left $\alpha$-equivalent} if $(g_1,f)=(g_2,f)$ for all $f\in \Fr_r(X)$. Denote by $A_{\ell}(X)$ the left $R$-module of equivalence classes. It is naturally a quotient of the free left $R$-module $\Fr_{\ell}(X)$. 

Two elements $f_1,f_2\in \Fr_{r}(X)$ are called \emph{right $\alpha$-equivalent} if $(g,f_1)=(g,f_2)$ for all $g\in \Fr_{\ell}(X)$. Denote by $A_{r}(X)$ the right $R$-module of equivalence classes. It is naturally a quotient of the free right $R$-module $\Fr_{r}(X)$.

The natural $R$-bilinear pairing 
\begin{equation}
    A_{\ell}(X) \times A_r(X) \lra R
\end{equation}
is nondegenerate. A morphism $f: X_1\lra X_2$ in $\mcC$ induces $R$-linear maps 
\begin{equation} \label{eq_lin_maps}
    A_r(f) : A_r(X_1) \lra A_r(X_2), \hspace{0.5cm} 
    A_{\ell}(f) : A_{\ell}(X_2) \lra A_{\ell}(X_1). 
\end{equation}
Maps $A_r(f)$, over all morphisms $f$ in $\mcC$, define a representation of $\mcC$, a covariant functor 
\begin{equation} 
A_{\alpha,r}=A_{r} \ : \ \mcC\lra \mathrm{mod-}R
\end{equation}
from $\mcC$ to the category of right $R$-modules. The functor assigns the state space $A_r(X)$ to an object $X$ and the map $A_r(f)$ to a morphism $f$. 

Maps $A_{\ell}(f)$, over all morphisms $f$ in $\mcC$, define a representation of $\mcC^{\op}$, a  functor 
\begin{equation}
  A_{\alpha,\ell} = A_{\ell} \ : \ \mcC^{\op}\lra R\dmod  
\end{equation} 
from the opposite category of $\mcC$ to the category of left $R$-modules. This functor assigns the state space $A_{\ell}(X)$ to an object $X$ and map $A_{\ell}(f)$ to a morphism $f$.   

We summarize this as the following statement. 

\begin{proposition}\label{prop_nonmon1} 
The universal construction for the evaluation $\alpha$ as above results in 
\begin{itemize}
\item
state spaces $A_{\ell}(X)$, $X\in\Ob(\mcC)$, which are left $R$-(semi)modules and, via maps \eqref{eq_lin_maps}, assemble into a contravariant functor $A_{\alpha,\ell}: \mcC^{\op}\lra R\dmod$, 
\item 
state spaces $A_{r}(X)$, $X\in\Ob(\mcC)$, which are right $R$-(semi)modules and, via maps \eqref{eq_lin_maps}, assemble into a 
covariant functor  $A_{\alpha,r}: \mcC\lra \mathrm{mod-}R$. 
\end{itemize} 
\end{proposition}

Informally speaking, we probe the category $\mcC$ via source objects $S_i$'s and target objects $T_j$'s and interpolate it via $\alpha$, looking at morphisms from $S_i$ to $T_j$ that factor through various objects $X$. Relations on morphisms to $X$ (and on morphisms from $X$) are introduced through the evaluation $\alpha$.  
When $R$ is a ring (rather than just a semiring), the construction is further simplified and $A_{\ell}(X), A_r(X)$ are (left, respectively right) modules over the ring $R$. 

\vspace{0.07in} 

\input{figure-P6-2}

\vspace{0.1in} 
 
 By analogy with Section~\ref{subsec_general} one can introduce the category $R\mcC$, whose morphisms are $R$-linear combinations of morphisms in $\mcC$. However, when $R$ is not commutative, functors $A_{\alpha,\ell}$ and $A_{\alpha,r}$ do not seem to extend to the category $R\mcC$. They can be extended to $Z\mcC$,  where $Z=Z(R)$ is the center of $R$. Category $Z\mcC$ has the same objects as $\mcC$, while morphisms in it  are finite $Z$-linear combinations of morphisms in $\mcC$. The composition extends $Z$-bilinearly from that in $\mcC$.
 
Evaluations $\alpha_{ji}$ in \eqref{eq_alphaji} extend $Z$-linearly to evaluations $\Hom_{Z\mcC}(S_i,T_j)\lra R$, also denoted $\alpha_{ji}$. 

Next, define $\alpha$-equivalence on morphisms as follows. Elements $f_1,f_2\in \Hom_{Z\mcC}(X_1,X_2)$ are \emph{$\alpha$-equivalent} if for any $i\in I,j\in J$ and morphisms $h_1: S_i\lra X_1,$ $h_2:X_2\lra T_j$, the following equality holds: 
\begin{equation}\label{eq_alpha_nc}
    \alpha_{ji}(h_2 f_1  h_1)\ = \ \alpha_{ji}(h_2 f_2 h_1). 
\end{equation}
Writing these elements as $Z$-linear combination of elements of $\Hom_{\mcC}(X_1,X_2)$,
\begin{equation}
   f_1=\sum_k z_{1,k} f_{1,k},  \hspace{0.5cm} 
   f_2=\sum_{k'} z_{2,k'} f_{2,k'},  \hspace{0.5cm}  
   z_{1,k},z_{2,k'}\in Z,  \hspace{0.5cm} 
   f_{1,k},f_{2,k'} \in \Hom_{\mcC}(X_1,X_2),
\end{equation}
we rewrite equation \eqref{eq_alpha_nc} as 
\begin{equation}\label{eq_alpha_nc_1}
    \sum_k z_{1,k} \,\alpha_{ji}(h_2 f_{1,k}  h_1) \ = \ \sum_{k'} z_{2,k'}\, \alpha_{ji}(h_2 f_{2,k'} h_1). 
\end{equation}
Thus, we compose $f_{1,k},f_{2,k'}$ with a morphism from one of the \emph{starting} objects $S_i$ and, on the other side, with a morphism into one of the \emph{ending} objects $T_j$ to get morphisms $ h_2 f_{1,k}  h_1$ and $ h_2 f_{2,k'} h_1$ from $S_i$ to $T_j$ and then evaluate via $\alpha_{ji}$.

\vspace{0.07in} 

Define $\mcC_{\alpha}$ to be the category with the same objects as $\mcC$, and $\Hom_{\mcC_{\alpha}}(X_1, X_2)$ are $\alpha$-equivalence classes of elements of $\Hom_{Z\mcC}(X_1,X_2)$, that is, finite $Z$-linear combinations of morphisms in $\Hom_{\mcC}(X_1,X_2)$ modulo the equivalence relation built from $\alpha$, see above. Functors $A_{\alpha,\ell}$ and $A_{\alpha,r}$ extend to functors on the category $\mcC_{\alpha}$.

\begin{example}
  Let $\mcC$ be a free category on one object $X$ and a finite set $\Sigma$ of generating morphisms. Then  $\End_{\mcC}(X)$ is the free monoid $\Sigma^{\ast}$ on $\Sigma$. We choose $X$ as the unique source and the unique target object. Finite length words $\omega=a_1\cdots a_n$, $a_i\in \Sigma$ in letters $\Sigma$, are in a bijection with endomorphisms of $X$, and the evaluation function $\alpha$ assigns $\lambda_{\omega}\in R$ to each $\omega\in R$. 
  
  The evaluation $\alpha$ is then encoded in the noncommutative power series, 
  \[
  \alpha=\sum_{\omega \in \Sigma^{\ast}} \lambda_{\omega} \omega, \hspace{0.50cm} 
  \lambda_{\omega}\in R,  
  \]
  with coefficients in $R$.
  
  The state space $A_{\ell}(X)$ is a left $R$-module and carries a right action of the monoid $\Sigma^{\ast}$. Using the opposite monoid $\Sigma^{\ast \op}$ of $\Sigma^{\ast}$, the two actions can be combined into a left action of the monoid algebra $R[\Sigma^{\ast \op}]$. 
  
The state space $A_{\ell}(X)$ is spanned by vectors $\langle \omega |$, over words $\omega\in \Sigma^{\ast}$. This space is a cyclic left $R[\Sigma^{\ast \op}]$-module with the initial vector $\langle \emptyset |$ and a nondegenerate trace map given by $\alpha$. 

If  $A_{\ell}(X)$ is a finitely-generated $R$-module, one says that $R$-valued language $\alpha$ is \emph{recognizable}. Then picking a free $R$-module cover $\widetilde{A}_{\ell}(X)\lra A_{\ell}(X)$ together with a lift of the cyclic vector $\langle \emptyset |$ gives 
a weighted automaton over $\Sigma$ for the $R$-valued language $\alpha$. We refer for details on weighted automata to~\cite[Section 1.6]{BR2} and~\cite{DKV09}. 

When $R=\Bool:=\{0,1|1+1=1\}$ is the Boolean semiring, one recovers both nondeterministic and deterministic finite state automata (FSA) from the state space $A_{\ell}(X)$, see~\cite{IK-top-automata}. The minimal deterministic FSA for the language $L=\alpha^{-1}(1)$ is given by the subset $Q\subset A_{\ell}(X)$ of \emph{pure} states, of the form $\langle \omega |$, where $\omega \in \Sigma^{\ast}$, rather than their Boolean combinations. Minimal nondeterministic FSA are given by taking minimal free $\Bool$-module covers of $A_{\ell}(X)$ and lifting the action of $\Sigma$ and the cyclic vector to them. 

More generally, when $R$ is a commutative semiring, the state spaces $A_{\ell}(X)$ can be extended from a single object $X$ to an entire rigid symmetric monoidal category of one-dimensional $\Sigma$-decorated cobordisms. This extension requires additionally choosing a circular $R$-valued series $\alpha_{\circ}$, to evaluate circles decorated by necklaces of letters from $\Sigma$, see~\cite{IK-top-automata,Kh3}. The original series $\alpha$ evaluates intervals decorated by chains of letters from $\Sigma$. This extension is only possible for commutative $R$, since intervals and circles give generators of the endomorphisms of the identity object of the one-dimensional cobordism category and commute (float past each other). For noncommutative $R$, only a non-monoidal construction above seems possible.
  
In~\cite{IK-top-automata} the universal construction was done on the monoidal category of one-dimensional cobordisms with defects, and the data for $\alpha$ required an additional circular series, also producing a monoidal category. Doing it in the non-monoidal setting, as above, requires only noncommutative power series, recovers the familiar notion of a nondeterministic automaton in the special case $R=\Bool$ but does not give a monoidal category. It does work for an arbitrary noncommutative semiring $R$, unlike the monoidal setting which requires $R$ to be commutative. 
  \end{example}

%%%%%%%%%%%%%%%%%%%%%%%%
%
% Brauer envelope 
%
%%%%%%%%%%%%%%%%%%%%%%%%

\section{Brauer envelopes of categories}
\label{sec_brauer}

%%%%%%%%%%%%%%%%%
% Brauer envelope 
%%%%%%%%%%%%%%%%

\subsection{Brauer envelope of a category.}
\label{subsec_brauer}
Starting with a small category $\mcC$, where objects constitute a set, we construct a rigid symmetric category $\BcC$ which we call  \emph{the Brauer envelope} of $\mcC$. The idea is to start with oriented one-manifolds, viewed as cobordisms between oriented zero-manifolds, decorate zero-manifolds by objects of $\mcC$ and one-manifolds by morphisms in $\mcC$. This further requires allowing dots (zero-dimensional defects) on one-manifolds labelled by objects of $\mcC$, while intervals between these defects are labelled by morphisms, in a compatible way. Such a defect with morphisms on both of its sides can be erased, with the pair of morphisms replaced by their composition. The construction also produces loop (circles) labelled by endomorphisms of objects in $\mcC$, modulo a suitable equivalence relation. 

\smallskip 

More carefully, objects of $\BcC$ are finite sequences $\undX=(X_1^{\varepsilon_1}, \dots, X_n^{\varepsilon_n})$ of pairs $X_i^{\varepsilon_i}=(X_i,\varepsilon_i)$ which are (object of $\mcC$, sign), where a sign $\varepsilon_i\in\{+,-\}$ is thought of as an orientation of a point.  The empty sequence $\emptyset_0$ is the identity object of $\BcC$. The category $\BcC$ is rigid symmetric monoidal with the following generating morphisms: 
\begin{itemize}
\item for each morphism $\beta\in \Hom_{\mcC}(X,Y)$ morphisms 
\[\beta_+:X^+\lra Y^+, \ \ \beta_-: Y^-\lra X^-,  \ \  \beta_{\capright}: X^+\otimes X^-\lra \emptyset_0, \ \ \beta_{\cupright}: \emptyset_0 \lra X^-\otimes X^+,
\] 

\item transposition morphisms 
\[ X^{\varepsilon_1}\otimes Y^{\varepsilon_2}\lra  Y^{\varepsilon_2}\otimes X^{\varepsilon_1}, \ \ \varepsilon_1,\varepsilon_2\in \{+,-\},
\]
\end{itemize} 
see Figure~\ref{figure-A1}. Implicitly, we assume that the two intervals in the diagram of a transposition morphism carry the identity morphisms of $X$ and $Y$.  We then define cup and cap morphisms for the opposite orientation as in Figure~\ref{figure-A2}. 

\input{figure-A1}

\vspace{0.1in} 

\input{figure-A2}

Monoidal composition of these generating morphisms results in diagrams of oriented arcs and circles (oriented one-dimensional cobordisms) with vertices labelled by objects of $\mcC$ and oriented edges labelled by morphisms in $\mcC$ from the source to the target vertex of the edge. The diagrams can be simplified using the composition in $\mcC$: a pair of composable morphisms $\beta:X\lra Y$, $\gamma: Y\lra Z$ corresponds to a decorated one-manifold with an inner vertex labelled $Y$ as in Figure~\ref{figure-A3} and a morphism $X_+\lra Z_+$. We impose the simplification relation allowing to erase that vertex and replace the pair of edges by one edge labelled by the composition $\gamma\beta$, and likewise for the opposite orientation. 

\vspace{0.1in} 

\input{figure-A3}

\input{figure-A4}

A similar simplification is introduced for vertices near cups and caps, see Figure~\ref{figure-A4} for an example. Intervals decorated by identity morphisms can be erased, see Figure~\ref{figure-A5} on the left. Figure~\ref{figure-A5} on the right shows a rigidity isomorphism, involving cup and cap labelled by the identity morphism. 

\input{figure-A5}

\vspace{0.07in} 

We impose the relations that 
 transposition morphisms together with the identity cup and cap morphisms $\id_{X,\cupleft},\id_{X,\cupright}$ and $\id_{X,\capleft},\id_{X,\capright}$ make $\BcC$ into a rigid symmetric monoidal category. In particular, this means that all intersections are ``virtual'', lines and circles can freely slide through each other and through dots (defects) on other lines and circles labelled by objects of $\mcC$, and all ``isotopies'' of diagrams are allowed. 

\begin{prop} Category $\BcC$ is a strict rigid symmetric monoidal category and contains $\mcC$ as a subcategory. The natural inclusion functor $\iota_{\mcC}: \mcC\lra \BcC$ is universal for functors from $\mcC$ to strict rigid symmetric monoidal categories. 
\end{prop} 

\begin{proof}  The faithful functor $\iota_{\mcC}:\mcC\lra \BcC$ assigns to $X\in \Ob(\mcC)$  the object $X^+$ of $\BcC$ and to a morphism $\beta$ in $\mcC$ the morphism $\beta_+$ in $\BcC$.  A functor $F:\mcC\lra \mathcal{D}$ to a strict rigid symmetric monoidal category $\mathcal{D}$ admits a canonical and unique extension to a rigid symmetric monoidal functor $\mathcal{B}(F):\BcC\lra \mathcal{D}$.  
\end{proof} 

Consider loops in $\BcC$. These are endomorphisms of the identity object $\one$ of $\BcC$ given by a circle with a finite set of objects $X_1,\dots, X_n\in \Ob(\mcC)$ placed on it in that order along the orientation direction and morphisms $\beta_i:X_i\lra X_{i+1},$ $X_{n+1}=X_1$ placed along the edges, see Figure~\ref{figure-A7}. 
Removing all but one vertex on a circle reduces us to an object $X\in \Ob(\mcC)$ with an endomorphism $\beta:X\lra X$.

\input{figure-A7}

\vspace{0.1in} 

 For a small $\mcC$ as above, define the set of loops $\LcC$ of $\mcC$ as the set of equivalence classes of pairs $(X,\beta)$, $\beta\in\End_{\mcC}(X)$ modulo the equivalence relation 
\begin{equation}
\label{equation:equivLcC}
(X,\gamma\beta)\sim (Y,\beta\gamma), \hspace{0.5cm} 
\mbox{for} \  \beta:X\lra Y, \hspace{0.5cm} 
\gamma:Y\lra X  
\end{equation} 
explained in Figure~\ref{figure-A6}. 
These equivalence classes correspond to equivalence classes of circles in $\End_{\BcC}(\one)$, which are generators of the latter commutative monoid. 

\begin{remark} The set $L(\mcC)$ can be thought of as the \emph{trace} of the category $\mcC$, see formula (1.1) in~\cite{BPW19}.
It is similar to the \emph{vertical trace} in monoidal categories, see~\cite[Section 5.5]{GHW22} and~\cite{BGHL14,BHLZ17,EL16}. 
 It can also be defined as coend:  
\[ 
L(\mcC)=\int^C\Hom(C,C), 
\] 
see \cite{Loregian21}. 
\end{remark} 

\input{figure-A6}

\begin{prop}  Endomorphisms of the identity object $\End_{\BcC}(\one)$ in $\BcC$ is a free commutative monoid on generators in the set $\LcC$. Composition of endomorphisms of $\one$ corresponds to the disjoint union of decorated circles, while the empty diagram corresponds to the identity endomorphism in $\End_{\BcC}(\one)$. 
\end{prop} 

\vspace{0.1in}

More generally, we can describe morphisms $\Hom_{\BcC}(\undX,\underline{Y})$ between arbitrary objects in $\BcC$. Using the rigid structure, it suffices to describe 
$\Hom_{\BcC}(\one,\undX)$, for an arbitrary sequence $\undX$. These homs are in a bijection with the following triples: 
\begin{itemize}
\item a sign-reversing pairing on elements in the sequence $\undX$, 
\item a morphism in $\mcC$ from the $-$ to the $+$ object in each pair,
\item an endomorphism of $\one$ in $\BcC$. 
\end{itemize} 
Figure~\ref{figure-B1} shows an example of such morphism. In general, a morphism in $\BcC$ is a collection of disjoint arcs and circles. Arcs constitute an orientation-respecting matching of terms in the tensor product of the source and target objects of the morphism, with a morphism in $\BcC$ assigned to each arc, and circles are loops in $\mcC$. 

\vspace{0.1in} 

\input{figure-B1}

\vspace{0.1in} 

{\it Poincar\'e dual description.} The Brauer category $\BcC$ can also be described via the Poincar\'e dual diagrammatics, where now objects correspond to arcs and morphisms to dots (to defects on lines and circles). Start with the presentation of a morphism of $\BcC$, place a dot in the middle of each interval and label it by the corresponding morphism. Then erase the original vertices at the endpoints and original dots inside the cobordism. Intervals in this Poincar\'e dual decomposition of the same one-manifold are labelled by objects of $\mcC$. An example of a morphism and its Poincar\'e dual presentation is shown in Figure~\ref{figure-B2}. 

\vspace{0.1in} 

\input{figure-B2}

\vspace{0.1in} 

Adjacent dots on an interval or circle may be merged to the composition of morphisms, see Figure~\ref{figure-B3}. 

\vspace{0.1in} 

\input{figure-B3}

\vspace{0.1in} 

A dot labelled by the identity morphism may be erased. With this convention, a dotless circle labelled by an object $X$ denotes the loop $(X,\id_X)$ in $\mcC$. 

\begin{remark}\label{remark_PD} 
This Poincar\'e dual description is the one used in the papers~\cite{Kh3,IK-top-automata,IK} when $\mcC$ has a single object with its endomorphism monoid being the free monoid $\Sigma^{\ast}$ on a finite set $\Sigma$. A morphism in $\mcC$ can then be written uniquely as a product of generating morphisms in $\Sigma$. In those papers, dots on lines and circles are labelled by elements of $\Sigma$, thus each line carries a morphism, the product of corresponding elements of $\Sigma$, and each circle carries a loop in $\mcC$, an element of the free monoid $\Sigma^{\ast}$ up to rotational equivalence. 
\end{remark} 

\begin{remark} In our construction of the Brauer envelope, we simply added the duals of all objects and allowed to ``bend'' morphisms to form the rigid symmetric monoidal closure  of $\mcC$. A more subtle question is addressed in~\cite{Coulembier21}, where the authors start with a monoidal (and not necessarily symmetric) category and study how to dualize one or more objects in it. 
\end{remark} 

%%%%%%%%%%%%%%
% UC for Brauer envelopes 
%%%%%%%%%%%%%%

\subsection{Universal construction for Brauer envelopes}
\label{subsec_uc_Brauer}
To do the universal construction for the rigid tensor category $\BcC$ we pick a commutative semiring $R$ and a multiplicative evaluation 
\begin{equation}\label{eq_alpha_ev}
\alpha: \End_{\BcC}(\one) \lra R. 
\end{equation} 
These evaluations are in a bijection with maps of sets 
\begin{equation}
\alpha : \LcC \lra R, 
\end{equation} 
where we specify the element of $R$ assigned to each loop in $\mcC$ (to each equivalence class of decorated circles in $\mcC$) and then extend multiplicatively to a homomorphism in \eqref{eq_alpha_ev} from the free monoid to $R$. See Figure~\ref{figure-B0}.

\input{figure-B0}

Evaluation $\alpha$ allows to do a universal construction on $\BcC$. One first passes to $R$-linear combinations of morphisms in $\BcC$ to get the category $R\BcC$. Then  one mods out by equivalence relations, see Case {\bf IV} in Section~\ref{subsec_general} and \eqref{eq_equiv_rsm}.  
The resulting category, denoted $\mcB_{\alpha}(\mcC)$, is a rigid symmetric monoidal $R$-linear category, with the same objects as $\BcC$. 

Thus, we obtain a diagram of four categories and three functors: 
\begin{equation}
\mcC \stackrel{\iota_{\mcC}}{\lra} \BcC \lra R\BcC \lra \mcB_{\alpha}(\mcC),
\end{equation} 
with the second, third and fourth categories rigid symmetric monoidal and the third and fourth categories, additionally, $R$-linear. 

\vspace{0.07in} 

{\it Functoriality.}  Given a functor $\mcF:\mcC_1\lra\mcC_2$, there is an induced rigid symmetric monoidal functor $\mcB(\mcF):\mcB(\mcC_1)\lra \mcB(\mcC_2)$. An $R$-evaluation $\alpha: L(\mcC_2)\lra R$ for the category $\mcC_2$ induces evaluation $\alpha_{\mcF}: L(\mcC_1)\lra R$ given by composing $\alpha$ with $\mcF$ applied to loops in $\mcC_1$. Functor $\mcF$ then induces an $R$-linear rigid symmetric functor, which can be denoted 
\[
\mcF_{\alpha}: \mcB_{\alpha_{\mcF}}(\mcC_1)\lra \mcB_{\alpha}(\mcC_2). 
\] 

\vspace{0.07in} 

{\it Examples.} As a special case, let $\mcC$ be a category with one object and one morphism (the identity morphism). Then $\BcC$ is the category of oriented 1-cobordisms between oriented 0-manifolds (with objects enumerated by sign sequences). The Brauer category in this case is given by picking $\lambda\in R$ (one assumes that $R$ is a field or, at least, a commutative ring) and evaluating a circle to $\lambda$. The resulting category is the usual \emph{oriented Brauer category}, see~\cite{Rey15}.

The quotient category $\mcB_{\lambda}(\mcC)$ of the oriented Bruaer category by the universal construction with circle evaluating to $\lambda$ is also known as \emph{the negligible quotient} of the Brauer category. 

\vspace{0.07in} 

To expand on Remark~\ref{remark_PD} above, 
if $\mcC$ has only one object, $\Ob(\mcC)=\{X\}$, the category $\mcC$ is described by the monoid $M=\End_{\mcC}(X)$ of endomorphisms of $X$. To match with the constructions of~\cite{Kh3,IK-top-automata,IK}, consider the Poincar\'e dual presentation of $\BcC$. All intervals in a morphism in $\BcC$ are labelled by the same object, so this labelling can be omitted. Dots (defects) are labelled by elements of $M$. If $\{m_i\}_{i\in I}$ are generators of $M$, one can reduce to labeling dots by $m_i$'s, subject to whatever relations hold in $M$. The identity morphism is unlabelled and does not require a dot.

Furthermore, if $M\cong \Sigma^{\ast}$ is a free monoid on a set $\Sigma$, then dots are decorated by elements of $\Sigma$, with no relations on concatenations of dots. 

\begin{itemize}
\item When $R$ is a field, the corresponding category is considered in~\cite[Section 2.4]{Kh3} (one further forms the Karoubi closure of $\mcB_{\alpha}(\mcC)$, see also~\cite{KS3}), as the universal construction for field-valued evaluations of circles with defects. This category (and its generalization when cobordism may have inner points, see Section~\ref{subsec_Brauer_boundary}) was further investigated in~\cite{IK}, also with $R$ a field. In particular, a rational evaluation $\alpha$ in this case gives rise to a symmetric Frobenius algebra and a two-dimensional TQFT restricted to thin flat surfaces.  
\item In~\cite{IK-top-automata} the authors consider the case when $R=\Bool=\{0,1|1+1=1\}$ is the Boolean semiring and pick evaluations of intervals and circles decorated by words, respectively cyclic words, in letters in $\Sigma$. This is given by a regular language $L_{\I}\subset \Sigma^{\ast}$ and a regular circular language $L_{\circ}$. A special case of that construction, when the language $L_{\I}$ is empty, results in the interpolation of the Brauer category as above via the circular language $L_{\circ}$ and produces a $\Bool$-linear rigid symmetric monoidal category with finite hom spaces, see~\cite{IK-top-automata}. For the more general case of an arbitrary regular $L_{\I}$ see also the comment following equation \eqref{eq_two_maps} in  Section~\ref{subsec_Brauer_boundary}. 
\end{itemize}

%%%%%%%%%%%%%%%%%%
% Brauer with boundary 
%%%%%%%%%%%%%%%%%%

\subsection{Brauer envelopes with boundary}
\label{subsec_Brauer_boundary} 
A left ideal $\mcI$ of a category $\mcC$ is a collection of morphisms in $\mcC$ closed under left composition with morphisms in $\mcC$: 
\begin{equation}\label{eq_def_left_id}
f\in \mcI, \ 
f \in \Hom_{\mcC}(X,Y), \hspace{0.5cm}
g\in \Hom_{\mcC}(Y,Z) \  
\Rightarrow \ gf \in \mcI.   
\end{equation} 
An example of a left ideal in a small category $\mcC$ is given by taking a collection $\mathcal{U}$ of objects of $\mcC$ and defining the ideal to consist of all morphisms in $\mcC$ with the source object in $\mcC$. Denote this ideal by  
\begin{equation}\label{eq_left_ideal}
\mcI_{\mathcal{U}}^0 := \bigcup_{X\in\, \mathcal{U}, Y\in \Ob(\mcC)} \Hom_{\mcC}(X,Y). 
\end{equation} 
Likewise, 
a right ideal $\mcI$ of a category $\mcC$ is a collection of morphisms in $\mcC$ closed under right composition with morphisms in $\mcC$: 
\begin{equation}\label{eq_def_right_id} 
f\in \mcI, \hspace{0.25cm}
f \in \Hom_{\mcC}(X,Y), \hspace{0.25cm}
g\in \Hom_{\mcC}(Z,X)
\ \Rightarrow \ fg\in \mcI.   
\end{equation} 
Take a collection $\mathcal{U}$ of objects of $\mcC$ and define the right ideal to consist of all morphisms in $\mcC$ with the source object in $\mcC$: 
\begin{equation} \label{eq_right_ideal} 
\mcI_{\mathcal{U}}^1 := \bigcup_{X\in \, \mathcal{U}, Y\in \Ob(\mcC)} \Hom_{\mcC}(Y,X). 
\end{equation} 

Suppose given a small category $\mcC$ with a left ideal $\mcI^{0}$ and a right ideal $\mcI^1$. 
Define a monoidal category $\mcB=\mcB(\mcC,\mcI^0,\mcI^1)$ as follows. 
It has the same objects of $\BcC$. 
The morphisms in $\BcC$ and their diagrammatics are enhanced by allowing one-dimensional cobordisms to end in the middle (to have boundary points that do not correspond to terms in the source or target object of a morphism). 
These endpoints are called \emph{inner} or \emph{loose} endpoints or boundaries of the  cobordism. 
The interval at an `in', respectively `out', inner boundary point is labelled by a morphism in the right ideal $\mcI^1$, respectively in the left ideal $\mcI^0$, see Figure~\ref{figure-B4}.   

\vspace{0.1in} 

\input{figure-B4}

Composition rules for morphisms extend from $\BcC$ to $\mcB$ as shown in Figure~\ref{figure-B5}. Endpoints of a one-cobordism can be bend arbitrarily, providing a rigid structure on $\mcB$. As for the category $\BcC$, the intersections are virtual and inner endpoints can slide through intersections. 

\vspace{0.1in} 

\input{figure-B5}

\vspace{0.1in}

\begin{prop} \label{prop_ideals}  For a small category $\mcC$ and left and right ideals $\mcI^0,\mcI^1$ in it, as above, the 
category $\BcC=\mcB(\mcC,\mcI^0,\mcI^1)$ is a strict rigid symmetric monoidal category and contains $\mcC$ as a subcategory. The endomorphism monoid $\End_{\BcC}(\one)$ of the identity object is a free abelian monoid on the disjoint union 
\[
\LcC \sqcup (\mcI^0\cap \mcI^1) 
\]
of the set of loops in $\mcC$ and the set of elements in the intersection $\mcI^0\cap \mcI^1$ of the two ideals. 
\end{prop} 

Indeed, floating intervals, together with loops in $\mcC$, freely generate $\End_{\BcC}(\one)$. Equivalence classes of floating intervals are parametrized by morphisms in $\mcI^0\cap \mcI^1$. Figure~\ref{figure-B6} shows a reduction of an arbitrary interval to the corresponding morphism. 

\vspace{0.1in} 

\input{figure-B6}

\begin{example} Morphisms from $\emptyset_0$ to the sequence $(X_1^+,X_2^-)$, where $X_1,X_2\in \Ob(\mcC)$, are of two types, shown in Figure~\ref{figure-B7}: 
\begin{itemize}
\item an arc from $X_2$ to $X_1$ labelled by a morphism $\beta:X_2\lra X_1$ together with an element of $\End_{\BcC}(\one)$ (this requires having such a morphism $\beta$), 
\item half-intervals with endpoints in $X_1^+$, $X_2^-$ labelled by morphisms $\beta_1\in \mcI^0,\beta_2\in \mcI^1$ with $\beta_1:Y_1\lra X_1$, $\beta_2:X_2\lra Y_2$, for some $Y_1,Y_2$, and an element of $\End_{\BcC}(\one)$ (if such morphisms exist).  
\end{itemize}

\input{figure-B7}

\end{example} 

To do the universal construction for $\BcC$, an evaluation $\alpha: \End_{\BcC}(\one)\lra R$ is determined by a map of sets 
\[
\alpha: \LcC \sqcup (\mcI^0\cap \mcI^1) \lra R. 
\]
Each such $\alpha$ gives rise to an $R$-linear rigid symmetric monoidal category $\mcB_{\alpha}(\mcC)=\mcB_{\alpha}(\mcC,\mcI^0,\mcI^1)$ with 
\[\End_{\mcB_{\alpha}(\mcC)}(\one)=R,
\]
that is, endomorphisms of the identity object form the ground commutative semiring $R$. 

Categories $\BcC=\mcB(\mcC,\mcI^0,\mcI^1)$ and $\mcB_{\alpha}(\mcC)=\mcB_{\alpha}(\mcC,\mcI^0,\mcI^1)$ may be called \emph{rook Brauer categories}, by analogy with \emph{rook Brauer algebras}~\cite{HdelM}. 

\vspace{0.1in}

Special cases of interpolation categories $\mcB_{\alpha}(\mcC,\mcI^0,\mcI^1)$ have been studied in~\cite{Kh3,IK-top-automata,IK}. In those papers $\mcC$ has a single object $X$ and the endomorphism ring $\End_{\mcC}(X)\cong \Sigma^{\ast}$, the free monoid on a finite set $\Sigma$, whose elements are called \emph{letters}. Sets $\mcI^0$ and $\mcI^1$ consist of all morphisms in $\mcC$ and are both parameterized by elements of $\Sigma^{\ast}$, that is, by \emph{words} in $\Sigma$. The set $\LcC$ of loops is parameterized by \emph{circular words} $\Sigma^{\circ}=\Sigma^{\ast}/\sim$, which are equivalence classes of words up to rotation, that is, $\omega_1\omega_2\sim \omega_2\omega_1$, for $\omega_1,\omega_2\in \Sigma^{\ast}$. Evaluation $\alpha$ is a map of sets 
\[
 \Sigma^{\ast}\sqcup \Sigma^{\circ} \lra R,
\]
and can be written via two maps: 
\begin{equation}\label{eq_two_maps}
 \alpha_\I: \Sigma^{\ast}\lra R, 
 \hspace{0.5cm} 
 \alpha_{\circ}: \Sigma^{\circ} \lra R,
\end{equation} 
that is,  as an assignment $\alpha_\I(\omega)\in R$ to each word $\omega$ and $\alpha_{\circ}(\omega)\in R$ to each circular word $\omega$. 

\begin{itemize}
 \item   In~\cite{Kh3} and~\cite{IK}, ring $R$ is a field $\kk$ and $\alpha_I,\alpha_{\circ}$ are \emph{noncommutative power series} in the set of variables $\Sigma$. State spaces for oriented $0$-manifolds are $\kk$-vector spaces, finite-dimensional if and only if  $\alpha_I,\alpha_{\circ}$ are \emph{rational} power series. For more on noncommutative rational power series, see~\cite{BR2,RRV} and references therein.
\item   In~\cite{IK-top-automata} $R$ is the Boolean semiring $\Bool=\{0,1|1+1=1\}$ and evaluations $\alpha_\I,\alpha_{\circ}$ are determined by subsets $\alpha_{\I}^{-1}(1)\subset \Sigma^{\ast}$ and $\alpha_{\circ}^{-1}(1)\subset \Sigma^{\circ}$, also called languages $L_{\I},L_{\circ}$, respectively. The state 
spaces for oriented $0$-manifolds are finite (equivalently, finite rank $\Bool$-modules) if and only if both languages $L_{\I},L_{\circ}$ are rational. A language $L\subset \Sigma^{\ast}$ is called \emph{rational}  if it is described by a regular expression or, equivalently, if it is the language accepted by a finite state automaton. Connections between universal construction and automata are explored in~\cite{IK-top-automata}, while an assignment of a one-dimensional defect TQFT to an nondeterministic automaton is explained in~\cite{GIKKL23}. 
\end{itemize} 
 
%%%%%%%%%%%%%%
%
%  Inner endpoints  
%
%%%%%%%%%%%%%%

\subsection{Presheaves of sets and inner endpoints}
\label{subsec_presheaves}

A more general setup for inner (floating) endpoints of intervals with the underlying category $\mcC$ is possible, as follows. Without specifying what data to put at an inner endpoint, consider a half-interval with a top $+$ boundary endpoint (and another, inner, endpoint). The boundary endpoint is labelled by some object $X\in \Ob(\mcC)$. A morphism $\beta:X\lra Y$ in $\mcC$ can be applied at the boundary, producing another half-interval, see Figure~\ref{figure-C1}.

\vspace{0.1in} 

\input{figure-C1}

These applications (maps) must respect composition of morphisms of $\mcC$. Thus, for each object $X$ there is a collection of all possible half-intervals with the $X$ label at their boundary and for each morphism in $\mcC$ a map between these collections, subject to the usual compatibility. We can describe this setup as a covariant functor  
\begin{equation}\label{eq_G0}
 \mcG_r: \mcC \lra \Sets 
\end{equation}
from $\mcC$ to the category of sets. Then, to a half-interval as above with $X$ at the boundary we can assign an element $g$ of the set $\mcG_r(X)$ (giving a decoration of that half-interval), and concatenation with the interval $\beta$ gives the element $\mcG_r(\beta)(g)$, see Figure~\ref{figure-C1} on the right. 

Reversing orientation and looking at oppositely-oriented half-intervals with $X$ at the boundary and their interactions via morphisms in $\mcC$, the resulting data can be described as a contravariant functor from $\mcC$ to the category of sets: 
\begin{equation}\label{eq_G1} 
\mcG_{\ell}: \mcC^{\op}\lra \Sets. 
\end{equation}
Such a functor $\mcG_{\ell}$ corresponds to a presheaf of sets on $\mcC$. Likewise, a functor $\mcG_r$ above a presheaf of sets on $\mcC^{\op}$ and can alternatively be called a \emph{precosheaf} of sets on $\mcC$. 

Composing two half-intervals results in a floating interval. 
To interpret this composition, form the quotient set 
\begin{equation} \label{eq_G_cross} 
\GCcross := \bigsqcup_{X\in \Ob(\mcC)} \mcG_{\ell}(X)\times \mcG_{r}(X)/\sim ,
\end{equation} 
with the equivalence relation generated by 
\[
(g_{\ell}\beta, g_r)\sim (g_{\ell},\beta g_r), \ X,Y\in \Ob(\mcC),\  \beta\in\Hom_{\mcC}(X,Y),\ g_{\ell}\in \mcG_{\ell}(Y), \ g_r\in \mcG_r(X). 
\]
This equivalence is shown diagrammatically in Figure~\ref{figure-C2}. Thus, the  general data we need to interpret endpoints is a choice of a presheaf of sets on $\mcC$ and on $\mcC^{\op}$ (functors $\mcG_{\ell}$ and $\mcG_r$ above). 

\input{figure-C2}

\begin{prop} Each datum $(\mcC,\mcG_{\ell},\mcG_r)$ as above gives rise to a rigid symmetric monoidal category $\mcB^{\I}(\mcC):=\mcB(\mcC,\mcG_{\ell},\mcG_r)$. In this category the endomorphism monoid $\End(\one)$ of the identity object is the free abelian monoid on the set 
\[
\LcC \sqcup (\GCcross). 
\]
\end{prop}
Note that $\End(\one)$ is freely generated by loops in $\mcC$ and floating intervals, and the latter are parametrized by elements of $\GCcross$.
An $R$-valued evaluation function $\alpha$ for the category $\mcB^{\I}(\mcC)$ is given by a map of sets 
\begin{equation}\label{alpha_boundary} 
\alpha: \LcC \sqcup (\GCcross) \lra R. 
\end{equation} 
Each such evaluation produces an $R$-linear rigid symmetric monoidal category, denoted $\mcB_{\alpha}^{\I}(\mcC)$. 

\vspace{0.1in} 

As a special case, one can take $\mcG_r$, respectively $\mcG_{\ell}$ to be a left ideal $\mcI^0$, respectively a right ideal $\mcI^1$ in $\mcC$, recovering the setup from Section~\ref{subsec_Brauer_boundary}. In slightly more detail, a left ideal $\mcI^0$ defines a covariant functor \[\mcG(\mcI^0):\mcC\lra \Sets
\]
that to an object $X$ of $\mcC$ assigns all $f\in \mcI^0$, $f\in \Hom(Y,X)$ for some $Y\in\Ob(\mcC)$. Functor $ \mcG(\mcI^0)$ is defined on morphisms via composition. We can then take $\mcG_r=\mcG(\mcI^0)$. Note the flip between left and right: a \emph{left} ideal $\mcI^0$ gives rise to the functor $\mcG_r$, which appears on the \emph{right} in the fibered product $\GCcross$.

Likewise, a right ideal $\mcI^1$ defines a functor $\mcG(\mcI^1):\mcC^{\op}\lra \Sets$, and we can set $\mcG_{\ell}=\mcG(\mcI^1)$. Again, there is a flip going from a right ideal to the functor $\mcG_{\ell}$, which appears on the \emph{left} in the fibered product $\GCcross$.  

\vspace{0.1in} 

{\it Non-monoidal case.}
Universal construction for a non-monoidal category $\mcC$ in Section~\ref{subsec_nonmon} can be extended to this setup as follows.  Start with a small category $\mcC$, pick a presheaf of sets $\mcG_{\ell}$ on $\mcC$, see \eqref{eq_G1} and a presheaf of sets $\mcG_r$ on $\mcC^{\op}$, see \eqref{eq_G0}. Pick a semiring $R$, not necessarily commutative. Consider the set $\GCcross$ and pick a map of sets 
\begin{equation} \label{alpha_again} 
\alpha \ :\ \GCcross\lra R.
\end{equation} 
The construction of Section~\ref{subsec_nonmon} can be extended to this setup. For an object $X\in \Ob(\mcC)$ define $\Fr_r(X)$ as the free right $R$-module with the basis $\mcG_r(X)$ and 
$\Fr_{\ell}(X)$ as the free left $R$-module with the basis $\mcG_{\ell}(X)$. Denote by $[g]$ the basis element associated with the element $g$ of $\mcG_r(X)$ or $\mcG_{\ell}(X)$. Elements of these modules can be written as $\sum_i  [g_{r,i}]a_i$, $g_{r,i}\in \mcG_r(X), a_i\in R$ and as $\sum_j b_j[g_{\ell,j}]$, $g_{\ell,j}\in \mcG_{\ell}(X), b_j\in R$, respectively. Map $\alpha$ defines an $R$-bilinear pairing
\begin{equation}
    (\:\:,\:\:)_N \ : \ \Fr_{\ell}(X)\times \Fr_r(X) \lra R, 
    \ \ \mbox{ where }
    \ \ ([g_r],[g_\ell]) := \alpha(g_r\times g_{\ell}). 
\end{equation}
On arbitrary linear combinations the map is 
\begin{equation}
\left(\sum_j b_j[g_{\ell,j}],\sum_i  [g_{r,i}]a_i\right) = \sum_{i,j} b_j \alpha(g_{\ell,j}\times g_{r,i})a_i. 
\end{equation} 
We can then define an equivalence relation on $\Fr_{\ell}(X)$ and on $\Fr_r(X)$ as in Section~\ref{subsec_nonmon}, identifying elements that give the same evaluation with any element of the opposite free module, and obtain the quotient $R$-modules: a right $R$-module $A_r(X):=\Fr_r(X)/\sim$ and a left $R$-module $A_{\ell}(X):=\Fr_{\ell}(X)/\sim$. The bilinear pairing 
\begin{equation}\label{eq_pairing2}
    A_{\ell}(X) \times A_r(X) \lra R
\end{equation}
is nondegenerate. A morphism $\beta:X\lra Y$ in $\mcC$ induces $R$-linear maps 
\begin{equation}\label{eq_maps_induce} 
A_r(\beta): A_r(X)\lra A_r(Y), \ \ A_{\ell}(\beta): A_{\ell}(Y)\lra A_{\ell}(X).
\end{equation} 
These maps, over all morphisms $\beta$ in $\mcC$, define a representation 
\[A_{\alpha,r}=A_r \ :\ \mcC\lra \mathrm{mod-}R,
\]
a covariant functor from $\mcC$ to the category of right $R$-modules, and a representation 
\[A_{\alpha,\ell} =A_{\ell}\ : \ \mcC^{\op}\lra R\dmod,
\]
a contravariant functor from $\mcC$ to the category of left $R$-modules. These functors should be viewed as the analogue of the universal construction in the nonmonoidal setting. Proposition~\ref{prop_nonmon1} extends to this setup as follows. 

\begin{proposition}\label{prop_nonmon2} Start with a small category $\mcC$, functors $\mcG_r:\mcC \lra \Sets $, $\mcG_{\ell}:\mcC^{\op} \lra \Sets,$     
a semiring $R$ and an evaluation map $\alpha:\GCcross\lra R$. 
The universal construction as described above results in 
\begin{itemize}
\item 
state spaces $A_{\ell}(X)$, $X\in\Ob(\mcC)$, which are left $R$-modules and, via maps \eqref{eq_maps_induce}, assemble into a contravariant functor  $A_{\ell} : \mcC^{\op}\lra R\dmod$. 
\item
state spaces $A_{r}(X)$, $X\in\Ob(\mcC)$, which are right $R$-modules and, via maps \eqref{eq_maps_induce}, assemble into a covariant functor $A_{\ell} : \mcC\lra \mathrm{mod-}R$. 
\end{itemize} 
The pairing \eqref{eq_pairing2} is nondegenerate. 
\end{proposition}

The setup of Section~\ref{subsec_nonmon} is recovered in the special case when one picks sets of objects $S_i,i\in I$ and $T_j,j\in J$ in $\mcC$ and specializes to functors 
\[\mcG_r(X):=\sqcup_{i\in I}\Hom_{\mcC}(S_i,X), \ \ \  \mcG_{\ell}(X):= \sqcup_{j\in J}\Hom_{\mcC}(X,T_j),
\]
that is, unions of representable functors, over $i$ and $j$ in those sets. In this case 
\[\GCcross=\bigsqcup_{i,j}\Hom_{\mcC}(S_i,T_j).
\]

%%%%%%%%%%%%%%
%
%  Pseudocharacters 
%
%%%%%%%%%%%%%%

\section{Pseudocharacters and one-dimensional topological theories}\label{sec-pseudo}

%%%%%%%%%%%%%%%%
% versus 
%%%%%%%%%%%%%%%%

\subsection{Topological theories versus TQFTs: a realization problem}
\label{subsec_realization} 

Let us fix a commutative semiring $R$ and a rigid symmetric monoidal category $\mcC$ (we restrict to that case, for simplicity). We distinguish between $R$-valued TQFT and an $R$-valued topological theory for $\mcC$, as follows: 
\begin{itemize}
\item 
An $R$-valued TQFT is a symmetric monoidal functor 
\begin{equation} \label{eq_func1} 
\mcF \ : \ \mcC\lra R\dfgpmod
\end{equation} 
from $\mcC$ to the category $R\dfgpmod$ of finitely-generated projective $R$-modules. 
\item 
A topological theory is a lax symmetric monoidal functor 
\begin{equation}
A_{\alpha} \ : \ \mcC \lra R\dmod 
\end{equation}
obtained via a universal construction from $\mcC$ for some evaluation 
\[\alpha:\End_{\mcC}(\one)\lra R
\]
(a monoid homomorphism that intertwines composition of endomorphisms with multiplication in $R$), see Section~\ref{subsec_general}.
\end{itemize} 

In a TQFT $\mcF$ there are natural isomorphisms 
\begin{equation}\label{eq_iso1} 
 \mcF(N_0)\otimes_R \mcF(N_1)\cong \mcF(N_0 \otimes N_1 ),
\ \ N_0,N_1\in \Ob(\mcC), 
\end{equation}
while in a topological theory there are only morphisms 
\begin{equation}\label{eq_only}
A_{\alpha}(N_0)\otimes_R A_{\alpha}(N_1) \lra A_{\alpha}(N_0\otimes N_1), \ \ N_0,N_1\in \Ob(\mcC), 
\end{equation} 
with suitable compatibility conditions. Equivalently, $A_{\alpha}$ gives a lax TQFT. Note that, in our approach, morphisms \eqref{eq_only} are not a part of the axiomatics; instead, they emerge from the universal construction for the evaluation $\alpha$.

\begin{remark} Replacing category $R\dfgpmod$ in \eqref{eq_func1} by the bigger category $R\dmod$ of all $R$-modules does not lead to any new functors. Due to the rigidity of $\mcC$ any object $X\in \Ob(\mcC)$ is necessarily mapped to a finitely-generated projective $R$-module, see~\cite{GIKKL23} for instance. 
\end{remark} 

\begin{remark} One can, more generally, consider \emph{weak topological theories} that are given by lax symmetric monoidal functors,  without the requirement that they come from an evaluation $\alpha$. This is a more flexible definition, but in the present paper we restrict to the more rigid case above. 
\end{remark}

\begin{definition} A \emph{realization} of an evaluation $\alpha$ (or a realization of topological theory) over a commutative semiring $R$ is a TQFT $\mcF$ over $R$ that, on endomorphisms of the identity object, restricts to the evaluation $\alpha$,
\[
\mcF(x) = \alpha(x), \ \ x\in\End_{\mcC}(\one). 
\]
\end{definition} 

%Variations of this definition are possible, where $\mcF$ may be defined over a larger commutative semiring $R'$ equipped with a homomorphism $R\lra R'$. 

Not all evaluations $\alpha$ come from TQFTs. For instance, each object $X$ of $\mcC$ defines an endomorphism $\SS_X$ of $\one$ (oriented circle labelled by $X$) given by composing two rigidity morphisms for $X$, see Figure~\ref{figure-D0}. Then $\alpha(\SS_X)=\rk(\mcF(X)),$ where $\rk(P)$ is the Hattori--Stallings rank of a finitely-generated projective $R$-module $P$, see~\cite[Section 2.1]{GIKKL23}, for instance. Also see \cite{Hatt65,Stall65,Han13,Kadison99}. In particular, if $R$ is a field $\kk$ and $\alpha$ on such endomorphisms does not take values in the image of $\Z_+$ in $\kk$, then $\alpha$ does not come from any TQFT $\mcF$ for $\mcC$. Indeed, in this case $\mcF(X)\cong \kk^n$, for some $n$, and $\alpha(\SS_X)=n$, viewed as an element of $\kk$.

\vspace{0.1in} 

\input{figure-D0}

\vspace{0.1in}

Assume from now on that $R$ is a commutative \emph{ring}.  Then for any $X\in \Ob(\mcC)$ we can form the  antisymmetrizer 
\begin{equation}
 e^-_{X,n} := \sum_{\sigma \in S_n} (-1)^{\ell(\sigma)}\sigma, 
 \hspace{0.75cm} 
 e^-_{X,n} \in \End_{\mcC}(X^{\otimes n}), 
\end{equation} 
where $\ell(\sigma)$ is the length  of the permutation $\sigma$ of the symmetric group $S_n$
and 
elements $\sigma$ act by permuting terms in $X^{\otimes n}$. Depict $e^-_{X,n}$ by a box labelled $n$ with $n$ incoming and $n$ outgoing edges, each labelled $X$, see Figure~\ref{figure-D1}. Note that 
\[
 e^-_{X,n} e^-_{X,n} = n! e^-_{X,n}, 
\]
and if $n!$ is invertible in $R$, the endomorphism $\frac{1}{n!}e^-_{X,n}$ can be defined, and is an idempotent. In the Karoubi envelope of $\mcC$, the object $\left(X^{\otimes n},\frac{1}{n!}e^-_{X,n}\right)$ can be interpreted as defining the $n$-th exterior power of $X$. 

\input{figure-D1}

\vspace{0.1in} 

For a TQFT $\mcF$ on $\mcC$, $R$-module $\mcF(X)$ is  finitely-generated and projective, thus it is a direct summand of $R^n$, for some $n$. Then $\mcF(e^-_{X,n})=0$, since $\Lambda^{n+1}(R^n)=0$. 

\vspace{0.07in} 

If an evaluation $\alpha$ comes from a TQFT $\mcF$, then for any $X$ there exists  $d\ge 0$ such that: 
\begin{itemize}
\item Any way to close up $e^-_{X,d+1}$ into an endomorphism of $\one$ evaluates to $0$ via $\alpha$, see Figure~\ref{figure-D2}. 
\item $\alpha$ evaluates some closure of $e^-_{X,d}$ to a nonzero element of $R$.  
\end{itemize} 
Note that $e^-_{X,0}$ is the identity endomorphism of $X^{\otimes 0}=\one$, and evaluates to $1\in R$ under $\alpha$.

\input{figure-D2}

\vspace{0.1in} 

Since $\mcC$ is rigid symmetric monoidal, any closure of $e^-_{X,n}$ can be presented as composing it with some endomorphism $h$ of $X^{\otimes n}$ and closing up the endpoints via $n$ concentric arcs, see Figure~\ref{figure-D2}. 
 
For an evaluation $\alpha$ and an object $X$, nonnegative integer $d$ that satisfies these conditions is unique if it exists. We set $\psdeg_{\alpha}(X)=d$ if such $d$ exists,  otherwise  define $\psdeg_{\alpha}(X)=\infty$. We then call $\psdeg_{\alpha}(X)$ the \emph{degree} of $X$ for the evaluation $\alpha$. 

\begin{definition}\label{def_pseudo} Suppose that $\Q$ is a subring of $R$.
    Evaluation $\alpha:\End_{\mcC}(\one)\lra R$ is a  \emph{pseudo-TQFT}  of $\mcC$ if $\psdeg_{\alpha}(X)<\infty$ for any object $X$ of $\mcC$.
\end{definition}

A \emph{pseudo-TQFT} may also be called a \emph{pseudocharacter} of $\mcC$ with values in $R$.

\begin{definition}\label{def_character} 
The \emph{character} of a TQFT $\mcF$ in \eqref{eq_func1} is the evaluation $\alpha_{\mcF}:\End_{\mcC}(\one)\lra R$ induced by $\mcF$.
\end{definition} 
 Alternatively, $\alpha_{\mcF}$ may be called the \emph{trace} of $\mcF$.
A character $\alpha_{\mcF}$ of a TQFT $\mcF$ is a pseudo-TQFT. One can ask: 

\begin{question}
Under what conditions on a monoidal category $\mcC$ and a commutative ring $R\supset \Q$ is any pseudo-TQFT for $\mcC$ over $R$ given by the character $\alpha_\mcF$ of some TQFT $\mcF$ for $\mcC$?
\end{question} 

\vspace{0.07in} 

To justify these definitions, start with a group or a monoid $G$ and associate to it one-object category $\mcC_G$ with object $X$ and $\End_{\mcC_G}(X)=G$. Consider the Brauer category $\mcB(\mcC_G)$ and possible evaluations $\alpha$ for it.  Loops in $\mcC_G$ are parameterized by conjugacy classes in $G$. For a monoid $G$ conjugacy classes are equivalence classes under the relation generated by $gh\sim hg$ for $g,h\in G$, see Figure~\ref{figure-D3}. 

\input{figure-D3}

\vspace{0.07in} 

An evaluation $\alpha:L(\mcC_G)\lra R$ is any $R$-valued function on conjugacy classes. Suppose 
 given a representation of $G$ on a finitely-generated $R$-module $V$. The evaluation $\alpha_V(g):=\tr_V(g)$ has additional property of being annihilated by some antisymmetrizer $e^-_{X,d+1}$, see earlier. Due to the structure of morphisms in the Brauer 
 category of $\mcC_G$ that condition can be rewritten as 
 \begin{equation} \label{eq_zero}
 \tr_{\alpha}((g_1\otimes \dots \otimes g_{d+1})e^-_{X,d+1})=0, \ \ \ 
 \mbox{ for each }\, g_1,\ldots, g_{d+1}\in G, 
 \end{equation} 
 and is shown in Figure~\ref{figure-D4}. 
 Let us specialize to $R$ with any projective module a free module (for instance, $R$ a field or a local ring). 
 Then Figure~\ref{figure-D4} condition holds for $d= \rk_R(V)$. If $\rk_R(V)!$ is invertible in $R$, that $d$ is minimal possible.  
 
 A conjugation-invariant function $\alpha$ on $G$ such that \eqref{eq_zero} holds for some $d$ is called  a \emph{pseudocharacter} of $G$. The degree $d$ of a  pseudocharacter $\alpha$ is the smallest nonnegative integer with property \eqref{eq_zero}. Any character of a representation of $G$ on a finitely-generated $R$-module is a pseudocharacter. 

 \begin{remark}
Examples of diagrammatic computation of some of these traces are shown in Figures~\ref{figure-D5} and \ref{figure-D6}, where all but one dots at the exits of the antisymmetrizer box are labelled by $x\in G$ and the remaining dot carries $y\in G$. In general, this diagram evaluates to $\alpha(P_{\alpha}(x)y)$, where $P_{\alpha}(Y)$ can be defined as the characteristic polynomial of a generic $d\times d$-matrix $Y$ with $\alpha(Y^k):=\tr(Y^k)$, see~\cite[Theorem 2]{Dot11} for the general statement.
\end{remark}

\vspace{0.1in} 

\input{figure-D4}

An important insight of Taylor~\cite{Taylor91}, Rouquier~\cite{Rou-pseudo96} and Nyssen~\cite{Nyssen96} (motivated by earlier work of Wiles~\cite{Wiles88} and Procesi~\cite{Proc76,Proce87}, see also Carayol~\cite{Carayol94} and Mazur~\cite[Chapter~2, \S~7]{Mazur97}) was that for some fields and local rings $R$, those conditions are enough: any $R$-valued pseudocharacter $\alpha$ on $G$ is the trace of a  representation $V$ of $G$.  

\begin{prop}[\cite{Taylor91,Rou-pseudo96,Nyssen96}]\label{prop_pseudo}
Let $R=\kk$ be a separably closed field, $\alpha:G\lra \kk$ a pseudocharacter of degree $d$ and assume that $d!$ is invertible in $\kk$. Then $\alpha$ is the trace (character) of a semisimple representation $V$ of $G$ over $\kk$ of dimension $d$. Representation $V$ is determined by $\alpha$ uniquely, up to isomorphism.
\end{prop} 

\begin{prop}[\cite{Taylor91,Rou-pseudo96,Nyssen96}] Let $R$ be a Henselian local ring, with a separably closed residue field $\kk=R/\mathfrak{m}$, and $G$ a monoid. Suppose given a pseudocharacter $\alpha:G\lra R$ of degree $d$ (with $d!$ invertible in $R$)  such that the reduction $\overline{\alpha}:G\lra \kk$ is irreducible, that is, not a sum of two non-trivial pseudocharacters. Then $\alpha$ is the trace of some representation of $G$ on $R^d$.  
\end{prop} 

We refer to Dotsenko~\cite{Dot11} and Bella\"iche~\cite{Bellaiche09} for an introduction and proofs of these results, in addition to the above original papers. 

The case important to number theory is when $G$ is the Galois group of some field $F$. 
A good motivation for the use of pseudocharacters in the theory of Galois representations can be found in the introduction to~\cite{Bella12}. For further developments and more applications we  refer the reader to~\cite{Bellaiche09, Bella12,Bella21} and~\cite{WW17,WWE18}. When $d!$ is not invertible in $R$, the definition of a pseudocharacter needs to be refined, see~\cite{Chene14,Emers18}. 

Johnson~\cite{Johnson19} covers other uses and applications of pseudocharacters, including work of Buchstaber and Rees on pseudocharacters for commutative rings and rings of continuous functions, see also~\cite{BR2004} and references therein.

\input{figure-D5}

\input{figure-D6}

\vspace{0.1in} 

We see that the notion of a pseudo-TQFT in Definition~\ref{def_pseudo} for a rigid symmetric monoidal category $\mcC$ is a natural generalization of the notion of a pseudocharacter of a group or a monoid $G$, both restricted to evaluations with values in commutative $\Q$-algebras $R$. Pseudocharacters of a monoid $G$ correspond to pseudo-TQFTs for the Brauer envelope $\mcB(\mcC_G)$ of the category $\mcC_G$ with a single object and its endomorphism monoid $G$. 

It is a very interesting problem to explore pseudo-TQFTs (and existence of liftings  to TQFTs) for categories $\mcC$ as above other than $\mcB(\mcC_G)$. A small step in this direction is done in the next section, where pseudo-TQFTs are considered for Brauer categories $\mcB(\mcC)$ where $\mcC$ has more than one object and for Brauer categories with inner endpoints (rook Brauer categories).

%%%%%%%%%%%%%%%%
% Distributed pseudocharacters 
%%%%%%%%%%%%%%%%

\subsection{Pseudocharacters of Brauer envelopes} 
\label{subsec_distributed}

We will use \emph{Brauer envelope} and \emph{Braeur category} interchangeably.

We would like to extend Proposition~\ref{prop_pseudo}, restricted for simplicity to a field $\kk$ of characteristic $0$, to pseudo-TQFTs over the Brauer category $\BcC$ of a (small) category $\mcC$. The idea of the extension is to replace a collection of objects of $\mcC$ by their direct sum. 

Let us first assume that  $\mcC$ has finitely many objects $X_1,\dots, X_n$. Choose a pseudocharacter $\alpha:L(\mcC)\lra R$ for a commutative ring $R$. 
We first linearize $\mcC$ and pass to $R\mcC$, the category with the same objects as $\mcC$ and morphisms -- $R$-linear combinations of morphism in $\mcC$. Pseudocharacter $\alpha$ extends linearly to a pseudocharacter, also denoted $\alpha$ on $R\mcC$. For $R$-linear categories, it is convenient to assume that a pseudocharacter is $R$-linear as well, which is true in our case. While many papers restrict to pseudocharacters given by maps $\alpha:G\lra R$, for a group or a semigroup $G$, Rouquier~\cite{Rou-pseudo96} replaces $R[G]$ by an arbitrary $R$-algebra, and its straightforward to extend Definitions~\ref{def_pseudo} and~\ref{def_character} from $\mcC$ to an $R$-linear category and replace the notion of a TQFT $\mcF$ for a rigid monoidal category $\mcC$ to that of a TQFT for such a category $\mcC$ which is, in addition, $R$-linear.  

Form the category $\mcC'$ by formally adding the object $X:=X_1\oplus \ldots \oplus X_n$ to $R\mcC$. Denote by $p_i:X\lra X_i$ the projection of $X$ onto $X_i$ and by $\iota_i:X_i \lra X$ the inclusion of $X_i$ into $X$.  Endomorphism $x=(x_{ij})$ of $X$ is a matrix of morphisms in $R\mcC$, with $x_{ij}\in\Hom_{R\mcC}(X_j,X_i)$. 
The morphism $\iota_ip_j$ is given by the elementary $(i,j)$ matrix in this presentation and projectors  $\iota_ip_i$ are elementary idempotent matrices. 

Consider the full subcategory of $\mcC'$ generated by the single object $X$, and denote this subcategory by $\mcC_X$. It is an $R$-linear category and its morphisms are $R$-linear combinations of various compositions of maps $\iota_ip_j$ and morphisms in $\Hom_{\mcC}(X_k,X_l)$. Evaluation $\alpha$ of $\mcC'$ restricts to an evaluation of the endomorphism algebra $\End_{\mcC_X}(X)$. 

Assume now that $R=\kk$ is a field of characteristic $0$. 

\begin{prop}\label{prop_when_pseudo}
The evaluation $\alpha$ is a pseudocharacter on $\mcC_X$ and 
\[
\deg_{\alpha}(X) =   \displaystyle{\sum_{i=1}^n} \deg_{\alpha}(X_i). 
\]
\end{prop} 
\begin{proof} It is enough to prove the proposition for $n=2$, with $X=X_1\oplus X_2$ and apply induction on $n$. Let $d_i=\deg_{\alpha}(X_i)$, $i=1,2$, and set $d=d_1+d_2$. Since $\id_X = \iota_1 p_1+\iota_2 p_2$, an endomorphism $h\in \End_{\mcC_X}(X^{\otimes (d+1)})$ can be written as $(\iota_1 p_1+\iota_2 p_2)^{\otimes (d+1)} \circ h \circ (\iota_1 p_1+\iota_2 p_2)^{\otimes (d+1)}$. Thus $h$ can be written as sum of terms $h_u$ which factor through $X_1$ or $X_2$ (through $\iota_1p_1$ or $\iota_2p_2$ at each \emph{in}  strand of $h$ (in the larger category $\mcC'$), see Figure~\ref{figure-E1}.

\vspace{0.1in} 

\input{figure-E1}

\input{figure-E2}

\input{figure-E3}

 Since $d=d_1+d_2$, in each term $h_u$ there are at least $d_1+1$ strands labelled $X_1$ or at least $d_2+1$ strands labelled $X_2$. Suppose that there are $t\ge d_1+1$ strands labelled $X_1$ in a given term $h_u$.  Conjugating by a permutation in $S_{d+1}$ these strands can be brought together and to the far left, see Figure~\ref{figure-E2}. 
The composition $h_u\circ e^-_{X,d+1}=h_u \sigma^{-1} \circ \sigma e^-_{X,d+1}$ factors through the map to $X_1^{\otimes t}\otimes X_2^{\otimes (d+1-t)}$, where $t$ is the number of $1$'s in the sequence $(i_1,\dots, i_{d+1})$.  

Permutation $\sigma$ is next to the sign idempotent and can be removed, at most at the cost of a sign, see Figure~\ref{figure-E3}. 

Next, a smaller sign  idempotent can be pulled from the one of size $d+1$:  
\begin{eqnarray*} 
((\iota_1 p_1)^{\otimes (d_1+1)}\otimes \id_X^{\otimes d_2}) e^-_{X,d+1}  & = & \small{\frac{1}{(d_1+1)!}}
\bigl( (\iota_1 p_1)^{\otimes (d_1 +1)}\otimes \id_X^{\otimes d_2}\bigr) 
\bigl( (e^-_{X,d_1+1}\otimes \id_X^{\otimes d_2}) \circ e^-_{X,d+1}\bigr)  \\
& = &  \frac{1}{(d_1+1)!}  
\bigl(\bigl(\iota_1^{\otimes (d_1+1)} \circ e^-_{X_1,d_1+1} \circ p_1^{\otimes (d_1+1)}\bigr) \otimes \id_X^{\otimes d_2}
\bigr)\circ e^-_{X,d+1}, 
\end{eqnarray*} 
see also Figure~\ref{figure-E3}. 

Since $t\ge d_1+1$, antisymmetrizer  $e^-_{X,d_1+1}$ in Figure~\ref{figure-E3} (when inserted into the appropriate position in Figure~\ref{figure-E2} on the right) has all top endpoints labelled $X_1$, so it can be rewritten as $e^-_{X_1,d_1+1}$. 

The latter idempotent, 
$e^-_{X_1,d_1+1}$,  will evaluate to $0$ when composed with any endomorphism of $X_1^{\otimes (d_1+1)}$ and evaluated upon closing up, since $\deg_\alpha(X_1)=d_1$. Consequently, closure of each term $h_u$ with $t\ge d_1+1$ evaluates to $0$. The same argument with $X_2$ in place of $X_1$ shows that each term $h_u$ with $t<g_1+1$ evaluates to $0$, so that $h$ evaluates to $0$. 

\vspace{0.07in} 

This implies that $\deg_{\alpha}(X_1\oplus X_2)\le d_1 + d_2$. Since $\kk$ has characteristic $0$, if $\alpha$ is a pseudocharacter, its degree $\deg_{\alpha}(X)=\tr_{\alpha}(\id_X)$, see~\cite[Proposition 3]{Dot11}. Hence 
\[
\deg_{\alpha}(X)=\tr_{\alpha}(\id_X) = \tr_{\alpha}(\id_{X_1})+\tr_{\alpha}(\id_{X_2})=d_1+d_2, 
\]
that is, $\alpha$ is a pseudocharacter on $\mcC_X$ of degree 
\[\deg_{\alpha}(X)= \deg_{\alpha}(X_1)+\deg_{\alpha}(X_2).
\]
\end{proof} 

Assume now that $\kk$ is an algebraically closed field of characteristic $0$. 
By the lifting property (Proposition~\ref{prop_pseudo}), $\alpha$ is a character of a $d_1+d_2$-dimensional semisimple representation $V$ of $\kk$-algebra $\End_{\mcC_X}(X)\cong \End(X_1\oplus X_2)$, where we assume that the algebra acts on  $V$ on the left. This representation is semisimple and the quotient algebra 
\begin{equation}\label{eq_quot_alg} 
B:=\End_{\mcC_X}(X)/\ker(\phi)
\end{equation} 
by the kernel of that action $\phi: \End_{\mcC_X}(X)\lra \End_{\kk}(V)$ is the direct product of finitely many matrix algebras, 
\begin{equation} \label{eq_prod} 
B\cong \prod_i\mathsf{Mat}_{n_i}(\kk).
\end{equation}

Endomorphisms $e_i := \iota_ip_i$ are mutually-orthogonal idempotents in $\End_{\mcC_X}(X)$ and  $1=e_1+e_2$. Let $V_i = e_i V$, $i=1,2$. A morphism $f\in \Hom_{\kk\mcC}(X_i,X_j)$ induces a $\kk$-linear map $V_i\lra V_j$, $e_iv_i \longmapsto fe_iv_i$, $v_i\in V_i$. This gives a representation $V$ of the categories $\mcC$ and $\kk\mcC$.  Here for simplicity we use the same notation $V$ for the representation of the algebra $\End_{\mcC_X}(X)$ and the category $\mcC$. 

Elements of $\End_{\kk\mcC}(X_i)$ act by zero on $V_j$, $j\not= i$, and their trace on $V_i$ is given by $\alpha$, $i=1,2$. Consequently, $\dim_{\kk}(V_i)=d_i$. 

The images $e_1',e_2'$ of $e_1,e_2$ in $B$ under the quotient map \eqref{eq_quot_alg} can be simultaneously conjugated inside each matrix algebra $\mathsf{Mat}_{n_i}(\kk)$ to the diagonal form, so that their images in $\mathsf{Mat}_{n_i}(\kk)$ (after changing a basis of $\kk^{n_i}$) are 
\[ 
e_1'=\sum_{j=1}^{m_i} e_{jj}, 
\hspace{0.75cm}
e_2'=\sum_{k=m_1+1}^{n_i} e_{kk}.
\] 
This explains the structure of quotient algebras $B_i:=\End_{\kk\mcC}(X_i)/\ker(\phi_i)$, where 
\[ 
\phi_i: \End_{\mcC}(X_i) \lra \End_{\kk}(V_i). 
\] 
Without loss of generality, assume that $n_i\geq m_i$. The corresponding representation $V_i$, which are semisimple representations of $B_i$ is: 
\begin{equation} 
\label{eq_prod_2} 
B_1\cong \prod_i\mathsf{Mat}_{m_i}(\kk), \ \ \ V_1\cong \bigoplus_i \kk^{m_i}, \ \ \ 
B_2\cong \prod_i\mathsf{Mat}_{n_i-m_i}(\kk), \ \ \  V_2\cong \bigoplus_i \kk^{n_i-m_i}. 
\end{equation} 
Module $V_1$ is a sum of column modules $\kk^{m_i}$ with matrix factors of $B_1$ acting on the corresponding summands, and likewise for $V_2$ and $B_2$. 

In particular, representation $V_i$ of $\End_{\kk\mcC}(X_i)$ and of its quotient algebra $B_i$ is semisimple. It is the unique, up to isomorphism, representation with the trace given by restricting $\alpha$ to endomorphisms of $X_i$ in $\mcC$ (or in $R\mcC$). 

Representation $V$ of $\mcC$ is semisimple and is determined uniquely, up to isomorphism, by $\alpha$. 

\vspace{0.1in} 

The above construction gives the following result. 

\begin{prop} \label{pseudo_prop_finite} Suppose $\mcC$ is a category with finitely many objects and $\alpha$ a pseudo-character of $\mcC$ valued in an algebraically closed field $\kk$ of characteristic $0$. Then $\alpha$ is the character of a semisimple representation $V$ of $\mcC$. Representation $V$ is unique up to an isomorphism. 
\end{prop} 

In particular, $V=\oplus_i V_i$, where $i$ parametrizes objects $X_i$ of $\mcC$. A morphism $f\in \Hom_{\mcC}(X_i,X_j)$ induces a map $V_i\lra V_j$, compatible with the composition of morphisms. The dimension $\dim V_i = \deg_{\alpha}(X_i)$ and $\alpha(f)=\tr_{V_i}(f)$, for $f\in \End_{\mcC}(X_i)$. Each representation $V_i$ of $\kk\End_{\mcC}(X_i)$ is semisimple. 

\vspace{0.1in} 

Next, we extend this proposition to $\mcC$ with countably many objects.

\begin{prop} \label{pseudo_prop_countable} Suppose $\mcC$ is a category with countably many objects and $\alpha$ a pseudo-character of $\mcC$ valued in an algebraically closed field $\kk$ of characteristic $0$. Then $\alpha$ is the character of a semisimple representation $V$ of $\mcC$. Representation $V$ is unique up to an isomorphism. 
\end{prop} 

\begin{proof} Enumerate objects of $\mcC$ by $X_1,X_2,\dots$. Consider the full subcategory $\mcC_N$ generated by objects $X_1,\dots, X_N$ and restriction $\alpha_N$ of $\alpha$ to $\mcC_N$.  There is a unique, up to isomorphism, semisimple representation $V^N$ of $\mcC_N$ with the trace $\alpha_N$. Adding the direct sum object $X=X_1\oplus \ldots\oplus X_N$, representation $V^N$ is a semisimple representation of $\kk\End(X)$. 

Form the direct sum $X\oplus X_{N+1}\cong \oplus_{i=1}^{N+1}X_i$. Pseudo-character $\alpha_{N+1}$ is the trace of a unique, up to isomorphism, semisimple representation $W$ of $\mcC_{N+1}$. Decompose $W\cong W_1\oplus W_2$, where $W_1$ is a representation of $\End(X)$ and $W_2$ a representation of $\End_{\mcC}(X_{N+1})$. Note that morphisms between $X$ and $X_{N+1}$ induce linear maps maps between $W_1,W_2$.

Fix an isomorphism $V^N\cong W_1$ of semisimple representations of $\mcC_N$. Via this isomorphism, we view $V^N\oplus W_2$ as a representation of $\mcC_{N+1}$. It is a semisimple representation with the trace $\alpha_{N+1}$. 

Let $G_N=\Aut_{\mcC_N}(V^N)$ be the group of automorphisms of the representation $V^N$ of $\mcC_N$. Restriction from $\mcC_{N+1}$ to $\mcC_N$ induces a homomorphism $\rho_N:G_{N+1}\lra G_N$.  

Starting with a fixed $N$ and continuing by induction on $N$, we obtain a representation $V$ of $\mcC$. It has a decomposition 
$V\cong \oplus_{i=1}^{\infty} V_i$, 
a morphism $f\in\Hom_{\mcC}(X_i,X_j)$ induces a $\kk$-linear map $V_i\lra V_j$, with these maps compatible with the composition of morphisms. Each $V_i$ is a finite-dimensional $\kk$-vector space and $\dim_{\kk}(V_i)=\deg_{\alpha}(X_i)$. Representation $V$ of $\mcC$ is semisimple and each $V_i$ is a semisimple representation of $\End_{\mcC}(X_i)$ with the trace given by restricting $\alpha$ to $\End_{\mcC}(X_i)$. Representation $V$ is a unique (up to isomorphism) semisimple representation of $\mcC$ with the trace given by $\alpha$.  
\end{proof} 

\begin{remark} 
Let $G_N=\Aut_{\mcC_N}(V^N)$ be the group of automorphisms of the representation $V^N$ of $\mcC_N$. Restriction from $\mcC_{N+1}$ to $\mcC_N$ induces a homomorphism $\rho_N:G_{N+1}\lra G_N$. The group of automorphisms of $V$, 
\[
\Aut_{\mcC}(V) \cong \lim_{\rho_N} G_N, 
\]
is isomorphic to the inverse limit of groups and homomorphisms 
\[
G_1 \stackrel{\rho_1}{\longleftarrow} G_2 \stackrel{\rho_2}{\longleftarrow} G_3\stackrel{\rho_3}{\longleftarrow} \dots 
\]
\end{remark}

 \begin{remark} Representations of $\mcC$ that we are considering may be called  \emph{left} representations of $\mcC$, being covariant functors $\mcC\lra \kk\mathsf{-fdvect}$ into the category of \emph{finite-dimensional} vector spaces. Right representations can be defined as contravariant functors from $\mcC$ to $\kk\mathsf{-fdvect}$. Composing either a covariant or a contravariant functor from $\mcC$ to $\kk\mathsf{-fdvect}$ with taking the dual of a finite-dimensional vector space induces a contravariant equivalence between categories of left and right representations of $\mcC$. It also gives a bijection between isomorphism classes of covariant and contravariant functors to $\kk\mathsf{-fdvect}$. This bijection respects semisimplicity. In particular, the propositions above hold for right representations of $\mcC$ as well. 
 \end{remark} 

 We do not expect difficulties with extending Proposition~\ref{pseudo_prop_countable} to small categories with uncountably many objects but will not attempt such an extension in the present paper. 

 %%%%%%%%%%%%%%%%
 % with endpoints
 %%%%%%%%%%%%%%%%%

\subsection{Pseudocharacters for Brauer categories with boundary} \label{subsec_pseudo_boundary}

Let us now discuss pseudocharacters for Brauer categories with endpoints. These categories were defined in Section~\ref{subsec_presheaves} (and a more restricted collection of examples--in Section~\ref{subsec_Brauer_boundary}). 
Let us now discuss pseudocharacters for Brauer categories with endpoints. These categories were defined in Section~\ref{subsec_presheaves} (and a more restricted collection of examples--in Section~\ref{subsec_Brauer_boundary}). 

Suppose given a datum $(\mcC,\mcG_{\ell},\mcG_r)$ as in Section~\ref{subsec_presheaves} of a small category $\mcC$ and functors 
\[
\mcG_{\ell}: \mcC^{\op}\lra \Sets , \ \ \mcG_r: \mcC \lra \Sets 
\]
in \eqref{eq_G0}, \eqref{eq_G1}. One can then form the rigid symmetric monoidal Brauer category with boundary $\mcB'(\mcC):=\mcB(\mcC,\mcG_{\ell},\mcG_r)$. 

Furthermore, pick an evaluation function $\alpha$ as in \eqref{alpha_boundary} valued in $R\supset \Q$,  
\begin{equation}\label{alpha_boundary_2} 
\alpha: \LcC \sqcup (\GCcross) \lra R. 
\end{equation} 

Suppose that evaluation $\alpha$ on endomorphisms of $\one$ in $\mcB'(\mcC)$ is a pseudo-TQFT or a pseudocharacter, in the sense of Definition~\ref{def_pseudo}. Then an object $X$ of $\mcC$ has some degree $d=\deg_{\alpha}(X)$. The condition that any closure of the $(d+1)$-antisymmetrizer $e^-_{X,d+1}$ vanishes upon applying $\alpha$, see Figure~\ref{figure-D2}, can be written for $\mcB'(\mcC)$ as follows.

Endomorphism $h$ of $X^{\otimes (d+1)}$ in Figure~\ref{figure-D2} can be written as a partial bijection of the set of endpoints $\{1,2,\dots, d+1\}$, where each arc and each half-interval in a bijection carry a label: an endomorphism of $X$ for an arc and an element of $\mcG_{\ell}(X)$ or $\mcG_r(X$ for a half-interval, see Figure~\ref{figure-BB1}. 

\vspace{0.1in} 

\input{figure-BB1}

Applying a permutation at a top or bottom of the antisymmetrizer $e^-_{X,d+1}$ changes it by at most a sign. Applying permutations to $h$ at the top and bottom it can be reduced to the form where all half-intervals are to the right of arcs and arcs are disjoint (define the identity permutation of $\{1,\dots, d+1-m\}$, where $m$ is the number of half-intervals at the top (and at the bottom), see Figure~\ref{figure-BB2}.  

\vspace{0.07in} 

\input{figure-BB2}

The same permutations, on the antisymmetrizer side, at most add the minus sign to it. 
Composing the resulting morphism $h'$ with the antisymmetrizer and closing up the composition result in the diagram in Figure~\ref{figure-BB3}.

\vspace{0.07in} 

\input{figure-BB3}

One can now expand the antisymmetrizer, fully or partially, to write it as the alternating sum of terms given by products of evaluation $\alpha$ on decorated circles and intervals. As an example, Figure~\ref{figure-BB4} shows an expansion for $d=2$ and $m=1$. 

\input{figure-BB4}

\vspace{0.07in}

Let us return to the Brauer category $\mcB'(\mcC)$ and further assume that pseudocharacter evaluation $\alpha$ takes values in an algebraically closed field $\kk$ of characteristic $0$. 

\begin{prop} \label{prop_brauer_ss}
Assume that $\mcC$ has finitely many objects and 
evaluation 
\begin{equation}\label{alpha_boundary_3} 
\alpha: \LcC \sqcup (\GCcross) \lra \kk, \ \ \Char(\kk)=0, \ \overline{\kk}=\kk,  
\end{equation} 
is a pseudocharacter. Then $\alpha$ is the  character of some semisimple representation $V$ of $(\mcC,\mcG_{\ell},\mcG_r)$.  The representation $V$ is unique, up to isomorphism. 
\end{prop} 

Note that a representation $V$ of $(\mcC,\mcG_{\ell},\mcG_r)$ consists of $\kk$-vector spaces $V_X$, for each $X\in\Ob(\mcC)$, linear maps $V_X\lra V_Y$ for each morphism $\beta\in\Hom_C(X,Y)$, a vector $v_{\gamma}\in V_X$ for each $\gamma\in \mcG_r(X)$, a covector $f_{\tau}\in V_X^{\ast}$ for each $\tau\in\mcG_{\ell}(X)$ subject to the standard compatibility relations on them given that $\mcC$ is a category and $\mcG_{\ell},\mcG_r$ are functors from it to $\Sets$. 

\begin{proof}
    Form the $\kk$-linear closure $\kk\mcC$ of  $\mcC$ and extend evaluation $\alpha$ $\kk$-linearly to it. This results in the rigid symmetric monoidal category $\mcB'_{\alpha}(\mcC)$.  The objects $X_1, \dots, X_k$ of $\mcC$ are naturally the objects of $\mcB'_{\alpha}(\mcC)$ as well. 
    
    Recall that the endomorphism ring of $\one$ in $\mcB'(\mcC)$ is generated by floating decorated circles and intervals. Circles may carry  dots (labeling morphisms in $\mcC$) and regions of the circle between the dots are colored by corresponding objects of $\mcC$. Endpoints of intervals are decorated by elements of $\mcG_{\ell}(X_i),\mcG_{\ell}(X_j)$, and intervals themselves are decorated by elements of $\GCcross$. 

  Upon evaluation $\alpha$, floating intervals and circles become elements of $\kk$, and
  $\End(\one)\cong \kk$ in the category  $\mcB'_{\alpha}(\mcC)$. 
  Let $X_0=\one$ be the identity object of $\mcB'_{\alpha}(\mcC)$. 

  We formally form the direct sum $X=X_0\oplus X_1\oplus \dots \oplus X_k$. It can be viewed as an object in the additive closure of $\mcB_{\alpha}'(\mcC)$ and we denote its endomorphism ring in the additive closure by $\End(X)$. Consider $\kk$-linear category $\mcC_X$ with a single object $\widetilde{X}$ and endomorphism ring $\End_{\mcC_X}\left(\widetilde{X}\right) = \End(X)$. 

Pseudo-character $\alpha$ of $\mcB'(\mcC)$ extends to a pseudo-character $\widetilde{\alpha}$ of $\mcB(\mcC_X)$. Evaluation $\widetilde{\alpha}$ is build from $\alpha$ as follows. Any endomorphism $t\in \End(X_i)$ in $\mcC$ gives an endomorphism $\widetilde{t}:=\iota_i\circ t\circ p_i$ of $\widetilde{X}$ by composing with the inclusion and projection $X_i \stackrel{\iota_i}{\lra} X$, $X\stackrel{p_i}{\lra} X_i$ and we define $\widetilde{\alpha}\left(\widetilde{t}\right) = \alpha(t)$.  

Endomorphisms $X_i\stackrel{u_i}{\lra} \one \stackrel{v_i}{\lra}  X_i$ that factor through $\one$ are evaluated to $\alpha(u_iv_i)$. Note that $u_iv_i$ is a linear combination of elements of $\GCcross$, and applying $\alpha$ to them produces elements in the ground field. Diagrammatically, $u\times v\in \GCcross$ is a floating interval, and we are, in a way, turning it into a circle by closing it up with a special line (dotted line in Figure~\ref{figure-BB5}) that represents the unit object $\one$. The evaluation of $u\times v$ is then viewed as that of a decorated circle, not a decorated interval. 

\vspace{0.1in} 

\input{figure-BB5}

\vspace{0.1in} 

Additional relations on dotted lines are shown in Figure~\ref{figure-BB5}. They follow from $\End(\one)\cong \kk$ in $\mcB'_{\alpha}(\mcC)$. 

We have 
\[
\deg_{\widetilde{\alpha}}(X) = \sum_{i=1}^k \deg_{\alpha}(X_i)+1,
\]
and $\widetilde{\alpha}$ is a pseudocharacter. By Proposition~\ref{pseudo_prop_finite} it is then the character of the unique (up to isomorphism) semisimple representation of $\mcB(C_X)$. Applying projections onto $X_i$, $i=0,1,\dots, k$ shows that $\alpha$ is the character of a unique (up to isomorphism) semisimple representation of $\mcB'(\mcC)$.    
\end{proof}

It is straightforward to extend this proposition to $\mcC$ with countably many objects.

%%%%%%%%%%%%%%%%%%%
%  pseudo-holonomies
%%%%%%%%%%%%%%%%%%%

\subsection{Pseudo-holonomies}\label{subsec_holonomies} 

 Let $M$ be a compact smooth connected $n$-manifold. To $M$ associate the category $\mcC_M$ with objects -- points of $M$. A morphism from $p_0$ to $p_1$ is a piecewise smooth path $\gamma:[0,1]\lra M$ with $p_0=\gamma(0)$ to $p_1=\gamma(1)$. These paths are considered up to backtracking, see Figure~\ref{figure-F1}, and reparametrizations (piecewise-smooth orientation-preserving homeomorphisms of $[0,1]$). 

\vspace{0.1in} 

\input{figure-F1}

Composition of morphisms is given by a composition of path. Category $\mcC_M$ is a groupoid, with the inverse of any path $\gamma$ being the reverse path $\overline{\gamma}$. This category is equivalent to its skeleton subcategory, given by picking a point $p\in M$ and restricting to the full subcategory $\mcC_M(p)$ of $\mcC_M$ consisting of endomorphisms of the object $p$. 

Suppose given a real vector bundle $E$ over $M$ and a connection $\nabla$ on $E$. It can be described by an $\R$-linear map 
\[
\nabla : \Gamma(E) \lra \Gamma(T^{\ast}M \otimes E) 
\]
where $\Gamma(E)$ is the space of smooth sections of $E$, and $\nabla(fs)=df\otimes s + f\nabla(s)$, for a smooth function $f$ and a smooth section $s$ of $E$. 

A connection induces a parallel transport map $\gamma_{\ast}: E_{\gamma(0)}\lra E_{\gamma(1)}$ of fibers  of the bundle along any piecewise-smooth path $\gamma$. A piecewise-smooth loop $\gamma:[0,1]\lra M$, $\gamma(0)=\gamma(1)$  induces a linear automorphism or holonomy of $\nabla$ along the path: $\gamma_{\ast}\in \mathsf{GL}(E_{\gamma(0)})$. For a closed path $\gamma$ trace of the holonomy $\tr_{\nabla}(\gamma,p)\in \R$ does not depend on the choice of a basepoint $p=\gamma(x),x\in[0,1)$ on the loop. Denote this trace by $\tr_{\nabla}(\gamma)$. It also only depends on the gauge equivalence class of the connection. 

The following proposition is clear. 

\begin{prop} For any connection $\nabla$ on  a bundle $E$, traces 
$\tr_{\nabla}(\gamma)$ over all loops $\gamma$ constitute an $\R$-valued pseudocharacter of the Brauer category $\mcB(\mcC_M)$ of degree $\dim(E)$.
\end{prop}

Vice versa, suppose given a pseudocharacter of $\mcB(\mcC_M)$ of degree $n$. Fixing a point $p\in M$ the pseudocharacter restricts to a pseudocharacter on the group $\End_{\mcC_M}(p)$. This pseudocharacter has degree $n$ and there exists a unique, up to isomorphism, semisimple representation $E_p$ of the group $\End_{\mcC_M}(p)$ of dimension $n$. Picking a path $\gamma_{q}$ from $q$ to $p$, for each point $q\in M$ allows to view $E_p$ as a representation $E_{q}$ of $\End_{\mcC_M}(q)$ and defines a semisimple representation of the Brauer category $\mcB(\mcC_M)$. 

\begin{remark} \label{rm_equivalence} More generally, an equivalence $F:\mcC_1\lra \mcC_2$ of categories induces a bijection of loops in these categories $L(\mcC_1)\cong L(\mcC_2)$ and a bijection between $\kk$-evaluations. This bijection preserves the pseudocharacter property. Thus, $F$ induces a bijection between pseudocharacters of Brauer categories $\mcB(\mcC_1)$ and $\mcB(\mcC_2)$. The present example corresponds to the equivalence of categories $\mcC_M(p)\lra \mcC_M$ given by including the full subcategory of endomorphisms of $p$ into the path category $\mcC_M$, the latter a groupoid.  
\end{remark} 

One can ask under what conditions can representations $E_{q}$, over $q\in M$, be glued into a bundle $E$ over $M$ with fibers $V_q$ and the action of the path category $\mcC_M$ coming from a connection $\nabla$. We refer to \cite{Lew93,CP94,BEP22}, and references therein for detailed studies of related questions. Here, we just observe the presence of a map

\begin{center}
\input{figure_gauge001}    
\end{center}

\noindent 
and ask for additional conditions to add on both sides to make this map a bijection. Notice that the map is not injective, in general. Pick a manifold with $\pi_1(M)$ which admits a nonsemisimple representation $\pi_1(M)\lra \mathsf{GL}(n)$ and form the corresponding bundle with a flat connection over $M$. Semisimple reduction of that representation produces a bundle with a flat connection not gauge equivalent to the original one. These two flat connections give rise to the same pseudocharacter of $\mcC_M$.

\begin{remark}
\label{remark:equiv_cat_endpoints} An equivalence of categories induces a bijection of loops in them and a bijection between their evaluations, see Remark~\ref{rm_equivalence}. For Brauer categories with endpoints, see Section~\ref{subsec_presheaves}, an equivalence 
$\mcC_1\lra \mcC_2$ does not induce a bijection on evaluations. For example, an interval may be labelled by an element of $\mcG_{\ell,2} \times_{\mcC_2}\mcG_{r,2}$ which is not pulled back to an element of $\mcG_{\ell,1} \times_{\mcC_1}\mcG_{r,1}$ via the equivalence, where $\mcG_{\ell,i},\mcG_{r,i}$ are suitable (contravariant, resp. covariant) functors from $\mcC_i$ to $\Sets$, $i=1,2$.   
Rather, one needs to have a bijection 
\[
\mcG_{\ell,1} \times_{\mcC_1}\mcG_{r,1} \cong 
\mcG_{\ell,2} \times_{\mcC_2}\mcG_{r,2}
\]
to get the corresponding bijection on evaluations or further quotient out these sets to have a bijection.  

\end{remark}

%%%%%%%%%%%%%%%%%%%
%
%  2D pseudo-TQFT
%
%%%%%%%%%%%%%%%%%%%

\section{Two-dimensional pseudo-TQFTs}\label{sec_2D_pseudo}

To understand two-dimensional pseudo-TQFTs let us first discuss 2D TQFTs and their generating functions. 

\subsection{2D TQFTs and generating functions} 
A 2D TQFT is a symmetric monoidal functor $\mcF:\Cob_2\lra \kk\mathrm{-vect}$ from the category of oriented 2D cobordisms to the category of vector spaces over a field $\kk$. Such TQFTs are determined by commutative Frobenius algebras $(B,\varepsilon)$ where $B=\mcF(\SS^1)$ is the commutative algebra that $\mcF$ associates to a circle and $\varepsilon:B\lra \kk$ is a nondegenerate trace on $B$ given by the \emph{cap} cobordism (a 2-disk viewed as a cobordism from $\SS^1$ to the empty one-manifold). 

A two-torus with one boundary component, see Figure~\ref{figure_handle001}  on the left, defines an element $h_B\in B$ which we call \emph{the handle element} of $(B,\varepsilon)$. This element gives rise to the handle map $B\lra B$ taking $u\in B$ to $h_B u$. This is the map induced by the cobordism in Figure~\ref{figure_handle001} on the right. 

Let $\{u_1,\dots, u_r\}$ be a basis of $B$ and $\{v_1,\dots, v_r\}$ the dual basis so that $\varepsilon(u_i v_j)=\delta_{i,j}.$ Then 
\begin{equation}
    h_B \ = \ \sum_{i=1}^r u_i v_i . 
\end{equation}

Let $\alpha_{B,g}=\mcF(S_g)\in \kk$ be the value of the oriented genus $g$ surface in the TQFT $(B,\epsilon)$, where we suppress the dependence of the value on the trace $\varepsilon$. It can be computed in two ways,  
\begin{equation}\label{eq_surface_eval}
    \alpha_{B,g} = \varepsilon(h_B^g), \ \ g\ge 0; \ \ \ \alpha_{B,g} = \tr_B(h_B^{g-1}), \ \ g\ge 1
\end{equation}
by applying the trace $\varepsilon$ to the $g$-th power of the handle element and as the trace of the operator $h_B^{g-1}$ acting on $B$, for $g\ge 1$. The second formula follows by representing $S_g$ as the torus with additional $g-1$ holes, see Figure~\ref{figure-Q1} on the right. Necessarily, 
\begin{equation}\label{eq_dimension}
 \alpha_{B,1} \ = \ \dim(B) ,  
\end{equation} 
that is, the torus evaluates to the dimension of $B$.

\vspace{0.07in} 

Define the \emph{generating function} of the TQFT $(B,\varepsilon)$ by combining the values  $\alpha_{B,g}$ into power series in one formal variable $T$: 
\begin{equation}
Z_{(B,\varepsilon)}(T) \ := \ \sum_{g\ge 0} \alpha_{B,g} T^g, \ \in \kk[\![T]\!]. 
\end{equation} 
It follows from the corresponding fact for topological theories~\cite{Kh2,KS2,KKO} that $Z_{(B,\varepsilon)}(T)$ is a rational function of $T$. Coefficient at the linear term $T$ is the dimension of $B$, see \eqref{eq_dimension}.  

Algebra $B$ is commutative and any idempotent $e$ in $ B$, $e^2=e$, gives a direct product decomposition $B\cong eB\times (1-e)B$ which respects the Frobenius algebra structure, since one can just take the components of $1$ and the trace map $\varepsilon$ in each term of the direct product. The comultiplication is determined by the trace and has a product decomposition as well. The generating function for $(B,\varepsilon)$ is the sum of generating functions for the direct summands $(eB,\varepsilon|_{eB})$, $((1-e)B,\varepsilon|_{(1-e)B})$. Direct product decompositions of commutative Frobenius algebras and more general TQFTs were considered in~\cite{Sawin95}. 

To do a more detailed analysis of possible generating functions, assume from now on that field $\kk$ is algebraically closed, $\overline{\kk}=\kk$. Consider the \emph{handle subalgebra} $\kkB $ of $B$ generated by $h_B$. The minimal polynomial $P_{h_B}(x)$ for the operator of multiplication by $h_B$ is the same whether one considers the action of $h_B$ on $B$ or on $\kkB$. 

Over $\kk=\overline{\kk}$ polynomial $P_{h_B}(x)$ factors into linear terms according to  the  eigenvalues of $h_B$. If there is more than one eigenvalue, subalgebra $\kkB$ admits a system of idempotents that allow to factorize $(B,\epsilon)$ into a direct product of Frobenius algebras $(B_i,\varepsilon_i)$ where in each algebra the handle element $h_{B_i}$ has a unique generalized eigenvalue $\lambda_i$ so that $h_{B_i}-\lambda_i\id$ is nilpotent. 

Let us assume that $(B,\varepsilon)$ is of that form, so that $h_B-\lambda \id$ is nilpotent on $B$. Then $\tr_B(h_B^k)=\dim(B)\lambda^k$ and the generating function has the form 
\begin{equation}
     Z_{(B,\varepsilon)} \ = \ \varepsilon(1) + \sum_{g\ge 1} \tr_B(h_B^{g-1})T^g = \varepsilon(1) + \dim(B)\sum_{g\ge 1}\lambda^{g-1} T^g .     
\end{equation}
It is natural to split into two cases: 
\begin{eqnarray} \label{eq_case_1} 
    Z_{(B,\varepsilon)} & = & \varepsilon(1) + \dim(B) T, \ \ \ \mathrm{if} \ \ \lambda=0, \\ 
    \label{eq_case_2} 
    Z_{(B,\varepsilon)} & = & \varepsilon(1) + 
    \frac{\dim(B) T}{1-\lambda T}, \  \ \ 
    \mathrm{if} \ \ \lambda\not= 0. 
\end{eqnarray}
 Note that $\dim(B)\ge 1$. If $\dim(B)=1$ then $\varepsilon(1)=\lambda^{-1}$ since the handle element $h_B=\lambda\not=0$, and the generating 
 function is 
 \begin{equation}
     Z_{(B,\varepsilon)} \ = \lambda^{-1} + \frac{T}{1-\lambda T}= \frac{\lambda^{-1}}{1-\lambda T}, \ \ \lambda\in \kk^{\ast}. 
 \end{equation}
Any finite-dimensional commutative $\kk$-algebra is a product of local algebras. If $B$ has idempotents other than $0,1$, Frobenius algebra $(B,\varepsilon)$ can be further factored. We can then assume that $B$ is a local ring, with a unique maximal ideal $\mathfrak{m}$. Necessarily, $B/\mathfrak{m}\cong \kk$ and there is a decomposition $B\cong \kk \,1 \oplus \mathfrak{m}$.  

\vspace{0.07in} 

Assume now that $\dim(B)\ge 2$. Choose a basis $\{1,u_2,\dots, u_r\}$ of $B$, $u_i\in \mathfrak{m}$, $r=\dim(B)\ge 2$, which determines the dual basis $\{v_1,\dots, v_r\}$, with the handle element 
\begin{equation}\label{eq_handle_basis}
h_B= v_1+\sum_{i=2}^r u_iv_i\in v_1 + \mathfrak{m}.
\end{equation} 

(1) Consider first the case $\lambda=0$, see \eqref{eq_case_1}. Then $h_B\in \mathfrak{m}$, since it acts nilpotently on $B$. 
We see that $v_1\in \mathfrak{m}$. 

Take $B=\kk[x]/(x^m)$ and 
\[
\varepsilon(x^i)=\begin{cases} 1, & i=m-1, \\
  0, & 1\le i \le m-2, \\
  \mu, & i=0,  
\end{cases} 
\]
so that $\varepsilon$ takes values $\mu$ and $1$ on $1$ and $x^{m-1}$, respectively, and $0$ on all other powers of $x$. Here $\mu\in \kk$ with no restrictions.  Clearly, $\varepsilon$ is nondegenerate since it is nonzero on the unique minimal ideal $(x^{m-1})$ of $B$ and turns $B$ into a Frobenius algebra. 
Take a basis $\{1,x,x^2,\dots, x^{m-1}\}$. Then the dual basis is $\{x^{m-1},x^{m-2},\dots,x, 1-\mu x^{m-1}\}$, which can be checked directly. 
The handle element $h_B= r x^{m-1}$ is indeed nilpotent, $h_B^2=0$. The generating function is 
\begin{equation}
    Z_{(B,\varepsilon)} \ = \ \mu + m T, \ \  \mu\in \kk, r\ge 2. 
\end{equation}

(2) Consider now the case $\lambda\not= 0$, see \eqref{eq_case_2}. As earlier, we assume that $B$ is local,  $B\cong \kk 1 \oplus \mathfrak{m}$, with $\mathfrak{m}$ the maximal ideal.

It is known that the handle element $h_B$ for a commutative local Frobenius algebra $(B,\varepsilon)$ with $\kk$ algebraically closed, lies in the socle of $B$, see the answer by user Mare in the discussion in 
\cite{Schommer-Pries22} and Propositions 3.6.14 and 1.10.18 in~\cite{Zimmermann14}.

In particular, $h_B^2=0$ and $h_B$ is nilpotent if $B$ is not the ground field $\kk$. Hence, if $\lambda\not=0$, algebra $B$ is the direct product of one-dimensional Frobenius algebras $\kk$, with $\varepsilon(1)=\lambda^{-1}$ in each term, and its generating function 
\[
   Z_{(B,\varepsilon)}(T) \ = \   \frac{\dim(B) \lambda^{-1}}{1-\lambda T}.  
\]

\begin{remark} 
 For a one-dimensional algebra $B=\kk$ with comultiplication $\Delta(1)=\gamma 1\otimes 1 $ and trace 
$\varepsilon(1)=\gamma^{-1}$, $\gamma\in \kk^{\ast}$ the value of the handle element is $\gamma$ and the generating 
function is 
\[
 \gamma^{-1} + T + \gamma T^2 + \dots \ = \ \frac{\gamma^{-1}}{1-\gamma T}. 
\]
Note that the coefficient at $T$ is one, which is the dimension of $B$. 
\end{remark} 

Putting this information together gives a complete answer for possible generating functions of commutative Frobenius algebras over  $\kk=\overline{\kk}$. 

\begin{prop}\label{prop_gen_func} Let $(B,\varepsilon)$ be a commutative Frobenius algebra over an algebraically closed field $\kk$. Then its generating function $Z_{(B,\varepsilon)}(T)$ has one of the following forms:
\begin{eqnarray}
\label{eq_Z_type1} 
 Z_{(B,\varepsilon)} & = & \mu + m T + 
    \sum_{i=1}^s\frac{m_i \lambda_i^{-1}}{1-\lambda_i T}, \ \ \ \mu\in\kk, \   m\in\{2,3,\dots\}, \
  m_i\in \{1,2,\dots\}, \ 
  \lambda_i \in \kk^{\ast}, \\
  \label{eq_Z_type2} 
   Z_{(B,\varepsilon)} & = &   \sum_{i=1}^s\frac{m_i \lambda_i^{-1}}{1-\lambda_i T},\ \ \ 
   m_i\in \{1,2,\dots\}, \ 
  \lambda_i \in \kk^{\ast},
\end{eqnarray}
and all possible values of the parameters are realized. 
\end{prop}

Notice that, in case \eqref{eq_Z_type1}, this generating function expands as 
\[
Z_{(B,\varepsilon)} \ = \ \left(\mu+\sum_{i=1}^s m_i\lambda_i^{-1}\right) + \left(m +\sum_{i=1}^s m_i\right)T + \dots 
\]
Any coefficient at the constant term can be realized via a suitable $\mu$, and $m\ge 2$ is required. 
If $\Char(\kk)=p$, coefficient $\alpha_1=\dim(B)$ of the generating function takes values in $\Z/p$, giving additional flexibility in picking $m$ if the rest of the parameters are fixed. 
Then \eqref{eq_Z_type2} is subsumed by \eqref{eq_Z_type2}, by replacing impossible values $m=0,1$ in \eqref{eq_Z_type1} by $m=0+p$, $1+p\ge 2$. 

\vspace{0.07in} 

We restrict our consideration to characteristic $0$ fields from now on,  where dimensions cannot be reduced modulo a prime. We then realize the above generating function by taking $B$ to be the direct product of the nilpotent algebra of dimension $m$ in example (1) above and $m_i$ copies of the one-dimensional Frobenius algebra with $\varepsilon(1)=\lambda_i^{-1}$, $1\le i\le s$.  

Let us emphasize that, specializing to $\Char(\kk)=0$, in $m\ge 2$ case the generating function has the most general form in \eqref{eq_Z_type1}, 
case $m=1$ is impossible (a Frobenius algebra with a nontrivial nilpotent ideal cannot have dimension $1$), and in $m=0$ case the coefficient at $T^0$ is determined by $m_i$'s and $\lambda_i$'s as 
$\sum_{i=1}^s m_i\lambda_i$. 

\vspace{0.1in} 

In case \eqref{eq_Z_type2}, which corresponds to a semisimple $B$ (product of $m_i$ copies of the one-dimensional Frobenius algebra with $\varepsilon(1)=\lambda_i^{-1}$, $1\le i\le s$), the generating function expands as 
\[
Z_{(B,\varepsilon)} \ = \ \left(\sum_{i=1}^s m_i\lambda_i^{-1}\right) + \left(\sum_{i=1}^s m_i\right) T + \dots 
\]
and there is no additional freedom to choose the constant term, unlike in \eqref{eq_Z_type1}. 

%%%%%%%%%%%%%%%
% 2D top theories 
%%%%%%%%%%%%%%%

\subsection{Two-dimensional topological theories.}
A topological theory $\alpha$  for oriented 2-dimensional cobordisms with values in a field $\kk$ is determined by the generating function of its evaluations 
\begin{equation}
Z_{\alpha}(T) \ := \ \sum_{g\ge 0} \alpha_g T^g, \ \  \alpha_g \in \kk. 
\end{equation} 
Here $T$ is a formal variable and $\alpha_g=\alpha(S_g)$ is the evaluation of a connected oriented surface of genus $g$. We refer the reader to~\cite{Kh2,KS3,KKO} for details on topological theories in two dimensions.

It was observed in~\cite{Kh2} that state spaces $A_\alpha(m)$ of one-manifolds $\sqcup_m\SS^1$, the union of $m$ circles, are finite-dimensional if and only if 
\begin{equation}\label{eq_Z_alpha} 
 Z_{\alpha}(T) = \frac{P(T)}{Q(T)}, \ \  Q(0)\not= 0,
\end{equation}
for some polynomials $P(T),Q(T)$, that is, if and only if $Z_{\alpha}(T)$ is a rational function in $T$. Let us assume that from now on (such topological theories are called \emph{rational}). 
A topological theory $\alpha$ gives  a \emph{lax} symmetric monoidal functor 
\begin{equation}
    \mcF_{\alpha} \ : \ \Cob_2 \lra \kk\mathrm{-vect} 
\end{equation}
taking $\sqcup_m\SS^1$ to its state space $A_\alpha(m)$ and inducing maps for oriented cobordisms between one-manifolds. 

One may ask under what conditions on $Z_{\alpha}(T)$ can this theory be lifted to a two-dimensional oriented TQFT over $\kk$, which were considered earlier in this section and which correspond to commutative Frobenius $\kk$-algebras $(B,\varepsilon)$. 

A lifting $\phi$ is a natural transformation 
\begin{equation}\label{eq_map_phi} 
\phi \ : \ \mcF_{\alpha} \lra \mcF_{(B,\varepsilon)}
\end{equation} 
from the lax monoidal functor $\mcF_{\alpha}$ to the monoidal functor $\mcF_{(B,\varepsilon)}$ that  preserves evaluations of closed surfaces: 
\begin{equation}
\mcF_{\alpha}(S) = \alpha(S), 
\end{equation} 
for all closed surfaces $S$. 

It consists of a 
collection of $\kk$-linear maps  $\phi_m : A_{\alpha}(m)\lra B^{\otimes m}$, for all $m\ge 0$, that intertwine all maps induced by two-dimensional oriented cobordisms between unions of $m$ and $m'$ circles. 
A closed surface $S$ is an endomorphism of the unit object $\emptyset_{1}$ of $\Cob_2$ (the empty one-manifold). Under functors $\mcF_{\alpha}$ and $\mcF_{(B,\varepsilon)}$ it goes to endomorphisms of multiplication by $\mcF_{\alpha}(S)$ and $\mcF_{(B,\varepsilon)}(S)$ of 
\begin{equation}
\kk= \mcF_{\alpha}(\emptyset_1)=  \mcF_{(B,\varepsilon)}(\emptyset_1).  
\end{equation} 

Necessarily, maps $\phi_m$ are inclusions, since the pairing of $A_{\alpha}(m)$ with itself (via the \emph{tube} cobordism) is nondegenerate. 

\vspace{0.07in} 

Since $\phi$ preserves evaluations of closed surfaces, the generating functions of the topological theory $\alpha$ and of the TQFT for the commutative Frobenius algebra $(B,\varepsilon)$ are equal: 
\begin{equation}
    Z_{\alpha}(T) \ = \ Z_{(B,\varepsilon)}(T). 
\end{equation}

\begin{corollary} 
A 2-dimensional topological theory $\alpha$ over an algebraically closed field $\kk$ that embeds into a TQFT over $\kk$ has  generating function $Z_{\alpha}(T)$ as in \eqref{eq_Z_type1} or \eqref{eq_Z_type2}.  
\end{corollary} 
This follows at once from 
proposition~\ref{prop_gen_func}. This corollary gives a \emph{necessary} condition for $\alpha$ to embed into a TQFT. Notice that equations \eqref{eq_Z_type1} and  \eqref{eq_Z_type2} are much more restrictive to  $Z_{\alpha}(T)$ than the rationality condition $Z_{\alpha}(T)=P(T)/Q(T)$, which is the criterion for $\alpha$ to have finite-dimensional state spaces. The latter condition, though, does not require any restrictions on the field.

\vspace{0.07in}

\begin{remark} 
In the topological theory for evaluation $\alpha$ the state space $A_{\alpha}(1)$ of one circle is also naturally a commutative Frobenius algebra, via the multiplication, unit and trace cobordisms. Thus, 
\[
\phi_1 \ : \ A_{\alpha}(1) \lra B 
\]
is a morphism of commutative Frobenius algebras. 

Algebra $A_{\alpha}(1)$ has a single generator, the handle element $h_{\alpha}$, see Figure~\ref{figure_handle001}. Multiplication by this element is the handle morphism $A_{\alpha}(1)\lra A_{\alpha}(1)$, see Figure~\ref{figure_handle001}. 
Likewise, $B$ has the handle element $h_B$ and $\phi_1(h_{\alpha}) = h_B$. We can then identify $A_{\alpha}(1)$ with the subalgebra $\kk\langle h_B\rangle $ of $B$ generated by $h_B$: 
\begin{equation}\label{eq_handle_alg}
A_{\alpha}(1) \cong \phi_1(A_{\alpha}(1)) = \kk\langle h_B\rangle \subset B. 
\end{equation} 
\end{remark}

\input{figure_handle001}

%%%%%%%%%%%%%%%%%%%%
% 2D pseudo TQFT
%%%%%%%%%%%%%%%%%%%%

\subsection{Pseudo-TQFTs in two dimensions.}
Assume now that evaluation $\alpha$ is a pseudocharacter of the category $\Cob_2$ of two-dimensional oriented cobordisms with values in an algebraically closed field $\kk$ of characteristic $0$. We use terms \emph{pseudocharacter} and \emph{pseudo-TQFT} interchangeably. 

\vspace{0.07in} 

Since $\Cob_2$ has a single generating object $\SS^1$ (the circle), a pseudocharacter $\alpha$ (or its generating function $Z_{\alpha}(T)$) has the property that for some $d\ge 0$ completing the antisymmetrizer $e^-_{X,d+1}$ to any closed cobordism and evaluating via $\alpha$ results in $0$. Smallest such $d$ is called the degree of $\alpha$ on $\SS^1$. 

Pick a cobordism $h$ from $d+1$ copies of the circle $\sqcup_{d+1}\SS^1$ to itself. Suppose that in $h$ two of circles on the same side of $h$ (either at the top or at the bottom) belong to the same connected component of $h$. Conjugating $h e^-_{\SS^1,d+1}$ or $e^-_{\SS^1,d+1} h$ by a permutation, if necessary, as in the proof of Proposition~\ref{prop_when_pseudo}, we can bring the two circles next to each other, see Figure~\ref{figure-R1}. 

\vspace{0.1in} 

\input{figure-R1}

Pulling out the antisymmetrizer of size two at these two circles from the larger antisymmetrizer $e^-_{\SS^1,d+1}$, see Figure~\ref{figure-R2} shows that either   $h e^-_{\SS^1,d+1}=0$ or $e^-_{\SS^1,d+1} h=0$, depending on whether the circles are at the bottom or the top. 

\vspace{0.1in} 

\input{figure-R2}

This reduces the pseudocharacter condition to cobordisms where all $d+1$ circles at the bottom of $h$ belong to different components of $h$, and likewise for the top $d+1$ circles. Further conjugating by permutations (due to the  antisymmetrizer these permutations at most change the sign of the evaluation of the closure of $he^-_{\SS^1,d+1}$), we can reduce the consideration to $h$ that consist of several vertical annuli on the left side, each possibly carrying one or more handles, and surfaces genus $0$ or higher capping off circles and the top and bottom on the right side, see Figure~\ref{figure-R3} on the left. 

\vspace{0.1in} 

\input{figure-R3}

These 2D cobordisms can be combinatorially encoded by 1D cobordism with defects, as shown in Figure~\ref{figure-R3} on the right. 

\vspace{0.1in} 

We see that the pseudocharacter condition on $\alpha$ needs only to be checked for special cobordisms that can be encoded by one-dimensional lines and half-intervals that carry defects of a single type, with no additional labels needed ($n$ defects on a line are encoded by a single dot with $n$ next to it). These one-dimensional cobordisms with defects can be interpreted as a special case of the Brauer category with inner endpoints as follows.

Consider the category $\mcC_1$ with a single object $X$ and a generating morphism $x:X\lra X$ with no relations on powers of $x$. Form the Brauer category $\mcB(\mcC_1)$. The latter is a category of one-dimensional oriented cobordisms carrying unlabelled dots. 

Next, we enhance $\mcB(\mcC_1)$ to a Brauer category with endpoints. We use covariant functor 
\begin{equation}\label{eq_funct1}
\Hom_{\mcC_1}(X,\ast) \ : \mcC_1 \lra \Sets, 
\end{equation}
respectively contravariant functor 
\begin{equation}\label{eq_funct2}
\Hom_{\mcC_1}(\ast,X) \ : \mcC_1^{\op} \lra \Sets, 
\end{equation} 
and form the Brauer category with endpoints $\mcB'(\mcC_1)$ for the category $\mcC_1$ and the above two functors. 

The resulting monoidal category has morphisms given by oriented 1D cobordisms between oriented 1D manifolds. Each connected component may have some number of dots. Floating components are oriented intervals and circles decorated by dots. Figure~\ref{figure-R4} gives an example of a morphism in  category $\mcB'(\mcC_1)$.  

\vspace{0.1in} 

\input{figure-R4}

There is a symmetric monoidal functor \begin{equation} \label{eq_func} 
\mcF_1 \ : \ \mcB'(\mcC_1)\lra \Cob_2
\end{equation} 
taking $X$ to an oriented circle and a vertical line with a dot to the handle morphism in Figure~\ref{figure_handle001}. The dual object $X^{\ast}$ is sent to the oppositely oriented circle. In particular, although $X$ and $X^{\ast}$ are not isomorphic in $\mcB'(\mcC_1)$, their images in $\Cob_2$ are isomorphic. 

\input{figure-Q1}

Functor $\mcF_1$ takes a circle with $n$ dots to a connected surface of genus $n+1$, see Figure~\ref{figure-Q1}. This functor also takes an interval with $n$ dots to a a connected surface of genus $n$. 

Suppose that $\alpha$ is a pseudocharacter on $\Cob_2$ evaluating a closed connected surface of genus $n$ to $\alpha_n$, $n\ge 0$. 

Pulling back $\alpha$ from $\Cob_2$ to $\mcB'(\mcC_1)$ via the functor $\mcF_1$ results in a pseudocharacter on $\mcB'(\mcC_1)$, denoted $\alpha^{\ast}$ that evaluates a circle with $n$ dots to $\alpha_{n+1}$ and an interval with $n$ dots to $\alpha_n$. 

From Proposition~\ref{prop_brauer_ss} we know that any pseudocharacter on $\mcB'(\mcC_1)$ with values in an algebraically closed field $\kk$ of characteristic $0$ comes from the character of a semisimple representation $V_1$ of $\mcC_1$ and the two functors \eqref{eq_funct1}, \eqref{eq_funct2}.

Representation $V_1$ assigns the ground field $\kk$ to the identity object $\one$ and a vector space $V$ to the object $X$. Oriented half-intervals induce maps $\iota$ and $p$ between $\kk$ and $V$ in the opposite directions, and endomorphism $x:X\lra X$ induces an endomorphism on $V$, denoted $h$, 
see Figure~\ref{figure-R5}.

\vspace{0.1in} 

\input{figure-R5}

\vspace{0.1in} 

We encode these two vector spaces and maps $\iota$, $p$, $h$ by a graph shown in Figure~\ref{figure-BB6} which has two vertices $v_0$, $v_1$, two oriented edges and one loop, at $v_1$. The data of $\kk,V$ and the three maps defines a particular representation of the corresponding quiver.

\input{figure-BB6}

\vspace{0.07in} 

The composition 
\[p\, h^n \iota: \kk \stackrel{\iota}{\lra} V\stackrel{h^n}{\lra} V \stackrel{p}{\lra} \kk
\]
is the scalar equal to the evaluation $\alpha^{\ast}$ of an oriented interval $I_n$ with $n$ dots. Functor $\mcF_1$ takes $I_n$ to $S_n$, the closed connected surface of genus $n$, so that 
\[
p\, h^n  \iota = \alpha^{\ast}(I_n)=\alpha(S_n)=\alpha_n. 
\]

Likewise, $\mcF_1$ takes a circle with $n$ dots to $S_{n+1}$, a genus $n+1$ connected surface, see Figure~\ref{figure-Q1}. In the TQFT $V_1$ circle with $n$ dots evaluates to the trace of the $n$-th power of $h$, so that $\tr_V(h^n)=\alpha_{n+1}$
Summarizing, we have the relations 
\begin{equation}\label{eq_two}
  p \, h^n \iota = \alpha_n, \ \  \tr_{V}(h^n) = \alpha_{n+1}, \ \ n\ge 0. 
\end{equation} 
Over algebraically closed field $\kk$ endomorphism $h$ of $V$ can be brought to an upper-triangular form. Assume that $h$ has eigenvalue $0$ with some multiplicity $m\in \Z_+=\{0,1,\dots\}$ and nonzero eigenvalues $\lambda_1,\dots, \lambda_s$ with multiplicities $m_1,\dots, m_s\in \{1,2,\dots\}$. Then 
\[
\tr_V(h^n) = \sum_{i=1}^s m_i \lambda_i^n 
\]
and 
\[
\sum_{n\ge 0} \alpha_{n+1} T^n =\sum_{n\ge 0} \tr_V(h^n) T^n =  m +  \sum_{n\ge 0}\sum_{i=1}^s m_i \lambda_i^n T^n = m + \sum_{i=1}^s \frac{m_i}{1-\lambda_i T}. 
\]
Additional term $m$ comes from the kernel of $h$, which contributes to the trace of $h^0$ but not to trace $h^n$ for $n\ge 1$. Note that traces of $h^n$ give no information on $\alpha_0$ (the evaluation of the 2-sphere), since only surfaces of genus $1$ and higher appear when applying $\mcF_1$ to a circle with dots. Multiplying the above by $T$ and adding $\alpha_0$ we obtain strong constraints on the generating function of $\alpha$: 

\begin{prop}\label{prop_gen_pseudo0} 
    Generating function of a pseudocharacter $\alpha$ on $\Cob_2$ with values in an algebraically closed field $\kk$ of characteristic $0$ is rational and has the following form: 
    \begin{equation}\label{eq_Zs}
        Z_{\alpha}(T) \ = \ \alpha_0 + mT+ 
    \sum_{i=1}^k \frac{m_i T}{1-\lambda_i T} = \mu + m T + 
    \sum_{i=1}^s\frac{m_i \lambda_i^{-1}}{1-\lambda_i T}, \ \   m\in\{0,1,\dots\}, \
  m_i\in \{1,2,\dots\}, \ 
  \lambda_i \in \kk^{\ast}, 
    \end{equation}
    where $\mu = \alpha_0 - \displaystyle{\sum_{i=1}^s} m_i\lambda_i^{-1}$. The pseudocharacter has degree 
    \[
\deg_{\alpha}(\SS^1)=m+\sum_{i=1}^s m_i,
\]
which is also the coefficient of the linear term in $Z_{\alpha}(T)$.  
\end{prop}
The case $s=0$, so that the generating function is linear, is allowed. 

\vspace{0.07in}

We will show that not all generating functions above come from pseudocharacters. 
 Notice first that the generating function of a 2D TQFT is necessarily the generating function of a pseudocharacter of $\Cob_2$. Generating functions of 2D TQFTs for algebraically closed fields $\kk$ are classified in Proposition~\ref{prop_gen_func}. Note that the classification is identical with that of Proposition~\ref{prop_gen_pseudo0} for $m\ge 2$. In particular, we obtain the following result. 

\begin{corollary}\label{cor_m_2} 
Any function as in \eqref{eq_Zs} with $m\ge 2$ is the generating function of a $\kk$-valued pseudocharacter of $\Cob_2$ with $\Char(\kk)=0$ and $\overline{\kk}=\kk$. 
\end{corollary} 

We now consider the remaining cases $m=0,1$. When $m=0$, generating functions of 2D TQFTs are more restrictive, see \eqref{eq_Z_type2}, while $m=1$ is not possible for a 2D TQFT. Since $m=\dim\, \ker(h)$,  these cases correspond to $h:V\lra V$ being invertible ($m=0$) and to $\ker(h)$ of dimension one  ($m=1$). 

\vspace{0.07in} 

Recall the relations \eqref{eq_two}, also reproduced below,
\begin{equation}\label{eq_two_again} 
  p \, h^n \iota = \alpha_n, \ \ \  \tr_{V}(h^n) = \alpha_{n+1}, \ \ \ n\ge 0. 
\end{equation} 
In particular, 
\begin{equation}\label{eq_equal} 
  p \, h^{n+1} \iota = \tr_{V}(h^n), \ \ n\ge 0. 
\end{equation} 
Since $p\, h^{n+1}\iota= \tr_{V}(\iota \,p  \, h^{n+1})$, we see that 
\begin{equation}\label{eq_trace2} 
 \tr_{V}(h^n(h\,\iota \, p -\id_{V})) =0, \ \ n\ge 0. 
\end{equation}

 Consider now the case $m=0$ in the generating function, that is, $\ker(h)=0$ and endomorphism $h$ is invertible. Write $h^{-1}=P(h)$ for some polynomial in $h$. Taking a linear combination of equations \eqref{eq_trace2} gives that 
\[
\tr_{V}(h^{-1}(h\,\iota \, p -\id_{V})) =0
\]
or that 
\[
\tr_{V}(\iota \, p)  = \tr_V(h^{-1}).
\]
Consequently, 
\[
\alpha_0 = \sum_{i=1}^s m_i \lambda_i^{-1}
\]
and, when $m=0$ (equivalently, when $h$ is invertible), the generating function of the pseudocharacter $\alpha$ has the form 
\[
 Z_{\alpha}(T) \ = \   \sum_{i=1}^s\frac{m_i \lambda_i^{-1}}{1-\lambda_i T},\ \ \ 
   m_i\in \{1,2,\dots\}, \ 
  \lambda_i \in \kk^{\ast}, 
\]
which is identical with possible generating functions of 2D TQFTs where the handle element $h$ is invertible, see Proposition~\ref{prop_gen_func} and \eqref{eq_Z_type2}. Thus, all such functions are indeed generating functions of pseudocharacters of $\Cob_2$.

\vspace{0.07in}

The remaining case is $m=1$. We treat it in several steps. 

\vspace{0.07in} 

{\bf I.} Consider first the case of diagonal $h$ with eigenvalues $0,\lambda_1,\dots, \lambda_N$, necessarily with $\lambda_i\not= 0$, $1\le i \le N$. In a suitable basis of $V$ we can write 
\begin{equation}
   p=\begin{pmatrix} 
   p_0 & p_1 & \dots & p_N  \end{pmatrix}, \ \ 
   h = \begin{pmatrix} 0 & 0 & \dots & 0 \\
0 & \lambda_1 & \dots & 0  \\
\vdots & \vdots & \ddots & \vdots \\
0 & 0 & \dots & \lambda_N \end{pmatrix}, \ \ 
   \iota = 
   \begin{pmatrix} 
   a_0 \\ a_1 \\ \vdots \\ a_N 
   \end{pmatrix}. 
\end{equation}
The relations are 
\begin{eqnarray} 
\alpha_0 & = & \sum_{i=0}^N p_i a_i, \\ 
\label{eq_alpha_n} 
\alpha_{n} & = & p_1a_1\lambda_1^{n}+\ldots + p_Na_N\lambda_N^n , \ \  n\ge 1, \\  
\label{eq_alpha_1} 
\alpha_1 & = & N+1, \\
\alpha_{n} & = & \lambda_1^{n-1} + \ldots + \lambda_N^{n-1}, \ \ n \ge 2.
\end{eqnarray} 
Equating the right hand side terms in the second and fourth relations for $n=2,3, \ldots, N+1$ and writing the resulting relations in the matrix form gives the equation  
\begin{equation}
    \begin{pmatrix} 
    \lambda_1^2 & \lambda_2^2 & \dots & \lambda_N^2 \\
    \lambda_1^3 & \lambda_2^3 & \dots & \lambda_N^3 \\
    \vdots  & \vdots & \ddots & \vdots \\
    \lambda_1^{N+1} & \lambda_2^{N+1} & \dots & \lambda_N^{N+1} 
    \end{pmatrix} 
    \begin{pmatrix} 
    p_1 a_1  \\ 
    p_2 a_2 \\ 
    \vdots \\ 
    p_Na_N      
    \end{pmatrix} = 
    \begin{pmatrix}
        \lambda_1 + \ldots + \lambda_N \\
        \lambda_1^2 + \ldots + \lambda_N^2 \\
        \vdots \\
        \lambda_1^{N+1} + \ldots + \lambda_N^{N+1}
    \end{pmatrix}
\end{equation}
The $N\times N$ matrix on the left hand side is the product of the diagonal matrix with entries $\lambda_i^2$ and the Vandermonde matrix. 
Assume that $\lambda_1,\dots, \lambda_N$ are distinct. Then this matrix is invertible and the system of equations has a unique solution 
\[
 p_i a_i = \lambda_i^{-1}, \ \ i=1,\dots, N. 
\]
Substituting this into the relation \eqref{eq_alpha_n} for $n=1$ gives $\alpha_1=N$, which contradicts the relation $\alpha_1=N+1$, see \eqref{eq_alpha_1}. This is a contradiction, assuming $\lambda_i$'s are pairwise distinct. 

To treat the case when some eigenvalues $\lambda_i$'s are equal, assume that the distinct eigenvalues among them are $\mu_1,\ldots, \mu_r$ and appear with multiplicities $m_1,\dots, m_r$ so that $m_1+\ldots + m_r = N$. Permute the basis vectors of $V$ so that the eigenvalues are $0, \mu_1,\ldots, \mu_1, \mu_2, \ldots, \mu_2, \mu_3, \ldots, \mu_r$, in this order, and the matrix of $h$ has the block form, $h=(0)\oplus \mu_1 \I_{m_1}\oplus \ldots \oplus \mu_r\I_{m_r}$, where $\I_m$ is the $m\times m$ identity matrix.  Computing the composition $p\, h^k\iota$, for $k>0$, we obtain 
\begin{eqnarray*}
 \alpha_k & = & p\, h^k \iota = p_1\mu_1^k a_1 +\ldots +p_N \mu_r^k a_N = (p_1a_1+\ldots + p_{m_1}a_{m_1})\mu_1^k+  \\
 &\hspace{2mm}& +\, (p_{m_1+1}a_{m_1+1}+\ldots + p_{m_1+m_2}a_{m_1+m_2})\mu_1^k + \ldots +( p_{N+1-m_r}a_{N+1-m_r}+\ldots + p_N a_N) \mu_r^k   \\
 & = & PA_1\mu_1^k+ PA_2 \mu_2^k + \ldots + PA_r \mu_r^k ,  
\end{eqnarray*}
where 
\[
PA_1 := p_1a_1+\ldots+ p_{m_1}a_{m_1}, \ \ 
PA_2 := p_{m_1+1}a_{m_1+1}+\ldots + p_{m_1+m_2}a_{m_1+m_2}, \ \ \ldots 
\]
Together with the relations  
\[
\alpha_{n+1}=\tr_V(h^n)=\sum_{i=1}^r m_i \mu_i^n,
\] 
equating the two expressions for for each of $\alpha_2,\dots,\alpha_{r+1}$ results in the matrix equation  
\begin{equation}
    \begin{pmatrix} \mu_1^2 & \mu_2^2 & \dots & \mu_r^2 \\
    \mu_1^3 & \mu_2^3 & \dots & \mu_r^3 \\
    \vdots  & \vdots & \ddots & \vdots \\
    \mu_1^{r+1} & \mu_2^{r+1} & \dots & \mu_r^{r+1} 
    \end{pmatrix} 
    \begin{pmatrix} 
    PA_1  \\ PA_2 \\ \vdots \\ PA_N     
    \end{pmatrix} = 
    \begin{pmatrix}
        m_1\mu_1 + \ldots + m_r\mu_r \\
        m_1\mu_1^2 + \ldots + m_r\mu_r^2 \\
        \vdots \\
        m_1\mu_1^{r} + \ldots + m_r\mu_r^{r}
    \end{pmatrix}.
\end{equation}  
The $r\times r$ matrix on the left hand side is nondegenerate since $\mu_i\not=0$ are pairwise distinct. Consequently, this system has a unique solution, which is easy to guess:
\[
PA_i = \frac{m_i}{\mu_i}, \ \ \ 1\le i\le r. 
\]
Writing down the two expressions for $\alpha_1$: 
\[
   p \, h \iota \ = \ \alpha_1 \ = \ \tr_V(h^0) \ = \ \dim(V)=N+1, 
\]
results in a contradiction 
\[
 p\, h \iota = \sum_{i=1}^r PA_i \mu_i = \sum_{i=1}^r m_i = N \not= N+1. 
\]
We see that  there are no solutions to the above system of equations \eqref{eq_two_again} when $h$ is semisimple with $\dim (\ker h)=1$. 

\vspace{0.07in} 

{\bf II.} Suppose next that $h$ has at least two nonsemisimple Jordan blocks with the same generalized eigenvalue $\lambda\not=0 $. Each of these blocks has a vector $v_i$, $i=1,2,$ satisfying $h v_i = \lambda v_i$. The map $p:V\lra \kk$, restricted to $\kk v_1\oplus \kk v_2$ has a kernel. Pick a nonzero vector $v$ in the kernel $p(v)=0$. Then $(0,\kk v)$ is a subrepresentation of $(\kk,V)$, since it is stable under the action of $p$, $\iota,$ and $h$. 

We can assume that $(\kk,V)$ is a semisimple representation of the quiver in Figure~\ref{figure-BB6}, by Proposition~\ref{prop_brauer_ss}. Then $(0,\kk v)$ has a complementary subrepresentation. Restricting to the action of $h$ results in a contradiction, since that one-dimensional semisimple representation in the sum of two nonsemisimple Jordan block cannot have a complement, even after a direct sum with any other finite-dimensional representation of $\kk[h]$. Note that semisimplicity condition on $(p,\iota,h)$ does not immediately imply semisimplicity for $h$, but does allow to exclude having a pair of nonsemisimple Jordan blocks for the same $\lambda$. 

\vspace{0.07in} 

We have now reduced to the case that $h$ has at most one nonsemisimple Jordan block for each of its eigenvalues. 

\vspace{0.07in}

{\bf III.}
Next, we treat the case when $h$ has a single Jordan block $J_{\lambda,N}$, besides the  $(0)$ summand: 
\begin{equation}
   p=\begin{pmatrix} 
   p_0 & p_1 & \dots & p_N  \end{pmatrix}, \ \ 
   h = (0)\oplus \begin{pmatrix} \lambda & 1 & 0 & \dots & 0 \\
0 & \lambda & 1 & \dots & 0  \\
 0 & 0 & \lambda & \dots & 0 \\
\vdots & \vdots & \vdots & \ddots & \vdots \\
0 & 0 & 0 & \dots & \lambda  \end{pmatrix}, \ \ 
   \iota = \begin{pmatrix} a_0 \\ a_1 \\ \vdots \\ a_N \end{pmatrix}. 
\end{equation}
Powers of $h$ are easy to write down: 
\begin{equation} \label{eq_h_power}
h^n  = (0)\oplus 
\renewcommand{\arraystretch}{1.35}
\begin{pmatrix} 
\lambda^n & n \lambda^{n-1} & \binom{n}{2} \lambda^{n-2} & \dots & 
\binom{n}{n-N+1} \lambda^{n-N+1} \\
0 & \lambda^n &  n \lambda^{n-1} & \dots & \binom{n}{n-N+2}\lambda^{n-N+2}  \\
 0 & 0 & \lambda^n & \dots &  \binom{n}{n-N+3} \lambda^{n-N+3} \\
\vdots & \vdots & \vdots & \ddots & \vdots \\
0 & 0 & 0 & \dots & \lambda^n  
\end{pmatrix},
\end{equation} 
with the convention that $\binom{n}{m}=0$ if $m<0$. 
Let 
\begin{equation}\label{eq_gammas} 
\gamma_i = \sum_{j=1}^{N+1-i} p_j a_{i+j} , \ \ 0\le i \le N-1
\end{equation} 
The equation \eqref{eq_equal}, also see below
\begin{equation}\label{eq_equal4} 
  p \, h^{n} \iota = \alpha_{n}=\tr_{V}(h^{n-1}), \ \ n\ge 2, 
\end{equation} 
for $n=2,3,\dots, N+1$ can be written as
\begin{equation}\label{eq_linear2} 
\renewcommand{\arraystretch}{1.35}
 \begin{pmatrix} 
 \lambda^2 & 2\lambda & 1 & \dots & 0 \\
\lambda^3 &  \binom{3}{1} \lambda^2 &  \binom{3}{2} \lambda & \dots & 0 \\
 \lambda^4  & \binom{4}{1}\lambda^3 & \binom{4}{2}\lambda^2 & \dots &  0 \\
\vdots & \vdots & \vdots & \ddots & \vdots \\
\lambda^{N+1} & \binom{N+1}{1}\lambda^{N} & \binom{N+1}{2}\lambda^{N-1} & \dots & \binom{N+1}{N-1}\lambda^2  
\end{pmatrix}
\begin{pmatrix} 
    \gamma_0  \\ \gamma_1 \\ \gamma_2 \\ \vdots \\ \gamma_{N-1}    
    \end{pmatrix} = 
    \begin{pmatrix}
        N\lambda \\
        N\lambda^2 \\
        N\lambda^3 \\
        \vdots \\
        N\lambda^{N}
    \end{pmatrix}.
\end{equation} 
Notice that the $N\times N$ matrix on the left hand side is given by taking the first column $(\lambda^2,\lambda^3,\dots, \lambda^{N+1})^T$, and the subsequent columns are scaled $j$-th derivatives $(j!)^{-1}d^j/d\lambda^j$ of that column, $j=1, \dots, N-1$. 

The $N\times N$ matrix $Y_N$ on the left hand side is invertible. One can see this by first subtracting $(N-1)$-st row times $\lambda$ from the last row, then subtracting $(N-2)$-nd row times $\lambda$ from row $N-1$, and eventually subtracting row one times $\lambda$ from the 2nd row. The resulting matrix has the form 
\[
\renewcommand{\arraystretch}{1.35}
\begin{pmatrix} 
\lambda^2 & 2\lambda & 1 &  0 & \dots & 0 \\
0 &   \lambda^2 &  2 \lambda &  1 & \dots & 0 \\
 0  & \lambda^3 & \binom{3}{1}\lambda^2 & \binom{3}{2}\lambda & \dots &  0 \\
\vdots & \vdots & \vdots & \vdots & \ddots & \vdots \\
0 & \lambda^{N} & \binom{N}{1}\lambda^{N-1} & \binom{N}{2}\lambda^{N-2} & \dots & \binom{N}{N-2}\lambda^2  
\end{pmatrix}.
\]
It is given by placing the truncated matrix $Y_{N-1}$ into the lower right corner and adding the first column $(\lambda^2,0,\dots, 0)^T$ and the first row $(\lambda^2,2\lambda,1,0,\dots, 0)$. 
Iterating this procedure we see that $\det (Y_N)=\lambda^{2N}\not=0$. Consequently, the above system of linear equations on $\gamma_0,\dots, \gamma_{N-1}$ has a unique solution. It is given by 
\[
\gamma_0= \frac{N}{\lambda}, \ \ \ \gamma_i =0, \ \ 1\le i\le N-1. 
\]
Writing down the equation 
\begin{equation}\label{eq_equal5} 
  p \, h \iota = \alpha_{1}=\tr_{V}(\I_{N+1})=N+1
\end{equation} 
gives $\lambda \gamma_0 = N = N+1$, a contradiction. 

\vspace{0.07in}
 
{\bf IV.} We now look at $h$ with one nonsemisimple $\lambda$-Jordan block and one or more semisimple (or $1\times 1$) $\lambda$-Jordan blocks:
\begin{equation}
   p=\begin{pmatrix} 
   p_0 & p_1 & \dots & p_M  \end{pmatrix}, \ \ \ 
   h = (0)\oplus J_{\lambda,N} \oplus (\lambda \I_1)^{M-N}, \ \ \
   \iota = 
   \begin{pmatrix} 
   a_0 \\ a_1 \\ \vdots \\ a_M 
   \end{pmatrix},  
\end{equation}
where $J_{\lambda,N}$ is the standard form of the Jordan block of size $N\times N$  and $(\lambda \I_1)^{M-N}$ is the identity matrix of size $M-N$ times $\lambda$. In that case $\gamma_i$ for $1\le i\le N$ are defined as in \eqref{eq_gammas} while 
\[
\gamma_0 = \sum_{j=0}^M p_j a_j 
\]
is lengthened by the extra terms corresponding to the semisimple tail of $h$. Then equation \eqref{eq_linear2} gets replaced by the same equation but with the column on the right hand side given by $(M\lambda,M\lambda^2,\dots, M\lambda^N)^T$, simply converting coefficient $N$ in each term to $M$. This  system has a unique solution 
\[
\gamma_0= \frac{M}{\lambda}, \ \ \gamma_i =0, \ 1\le i\le N-1. 
\]
Inserting these values into the equations for $\alpha_1$ again results in the contradiction, giving $M=M+1$. 

\vspace{0.07in} 

{\bf V.} Assume now that $h$ consists of the $(0)$ block and several nonsemisimple Jordan blocks $J_{\lambda_i,N_i}$, $N_i>1$,  with pairwise distinct eigenvalues $\lambda_i$: 
\[
h = (0)\oplus J_{\lambda_1,N_1}\oplus \ldots \oplus J_{\lambda_r,N_r}, \ \  N_1+\ldots +N_r =N. 
\]
For each pair $(\lambda_i,N_i)$ set up an $N\times N_i$ matrix $T_i$ with the first column $(\lambda_i^2,\lambda_i^3,\dots, \lambda_i^{N+1})^T$ and subsequent columns given by taking normalized $j$-th derivatives $(j!)^{-1}d^j/d\lambda^j$ of entries of the first column, $j=1, \dots, N_i-1$, 
\begin{equation} \label{eq_linear5} 
    T_i = 
\renewcommand{\arraystretch}{1.35}
 \begin{pmatrix} 
 \lambda_i^2 & 2\lambda_i & 1 & \dots & 0 \\
\lambda_i^3 &  \binom{3}{1} \lambda_i^2 &  \binom{3}{2} \lambda_i & \dots & 0 \\
 \lambda_i^4  & \binom{4}{1}\lambda_i^3 & \binom{4}{2}\lambda_i^2 & \dots &  0 \\
\vdots & \vdots & \vdots & \ddots & \vdots \\
\lambda_i^{N+1} & \binom{N+1}{1}\lambda_i^{N} & \binom{N+1}{2}\lambda_i^{N-1} & \dots & \binom{N+1}{N_i-1}\lambda_i^{N-N_i}  \end{pmatrix}.
\end{equation}
Next, put these matrices together into an $N\times N$ matrix
\begin{equation}\label{eq_T}
    T = \begin{pmatrix} 
    T_1 & T_2 & \dots &  T_r 
    \end{pmatrix}. 
\end{equation}
For each Jordan block $J_{\lambda_i,n_i}$ define 
\begin{equation}\label{eq_gammas2} 
\gamma_{i,j} = \sum_{s=1}^{N_i+1-j} p_{i,s} a_{i,j+s} , \ \ 0\le s \le N_i -1,
\end{equation} 
where 
\[
p_{i,s} = p_{N_1+\ldots +N_{i-1}+s}, \ \ \  a_{i,j+s} = a_{N+1+\ldots + N_{i-1}+j+s}. 
\]
Define $N\times 1$ matrix $\Gamma$ by 
\[
\Gamma^T=\begin{pmatrix}
    \gamma_{1,0} & \ldots & \gamma_{1,N_1-1} & \gamma_{2,0} & \ldots & \gamma_{2,N_2-1} & \ldots & \gamma_{r,0}  & \ldots & \gamma_{r,N_r-1}
\end{pmatrix}.
\]
Finally, consider $N\times 1$ matrix 
\[
R^T = \begin{pmatrix}
    \displaystyle{\sum_{i=1}^r} N_i \lambda_i &  
    \displaystyle{\sum_{i=1}^r} N_i \lambda^2_i & 
    \ldots & 
    \displaystyle{\sum_{i=1}^r} N_i \lambda_i^{N}
\end{pmatrix}.
\]
The equalities coming from the two expressions for $\alpha_2,\ldots, \alpha_{N+1}$ can be written as the matrix equation 
\begin{equation}\label{eq_TG}
 T \, \Gamma \ = \ R.
\end{equation} 
\begin{lemma} \label{lemma_det_form} 
The determinant of $T$ has the form 
\[
\det \, T \ = \ u \, \prod_{i=1}^r \lambda_i^{2N_i} \cdot \prod_{i<j} (\lambda_i-\lambda_j)^{N_iN_j},
\]
where $u$ is an invertible rational number. 
\end{lemma} 
This lemma is proved in~\cite{MatSep2018}. The matrix $\mathcal{U}[0]$ in Section 5.1.1 of that paper is a generalization of the matrix $T$ above. One needs to scale columns of $\mathcal{U}[0]$ by factorials to make the first row consist of powers of $\alpha_1$, in the notations of~\cite{MatSep2018}, and then scale rows by factorials as well to make entries into products of binomial coefficients times $\alpha_i^j$. Rotating the matrix 90 degrees counterclockwise,  specializing $\ell_0=2$ and relabelling $\alpha_i$ to $\lambda_i$ and $L_0$ to $N+2$ recovers the matrix $T$ above. 
Lemmas 5.2, 5.3 and the proof of Lemma 5.4 in~\cite{MatSep2018} imply that the determinant of $T$ has the form  as in Lemma~\ref{lemma_det_form}.
$\square$

In our experiments, parameter $u=\pm 1$. 

\vspace{0.07in} 

Since $\lambda_i\not=0$ and $\lambda_i$'s are pairwise distinct, we see that the equation \eqref{eq_TG} has a unique solution, which can be guessed to be 
\begin{equation}\label{eq_gamma_3}
\gamma_{i,0} = \frac{N_i}{\lambda_i}, \ \
\gamma_{i,j}=0, \ j>0, \ \ 1\le i \le r. 
\end{equation}
Equating the two expressions for $\alpha_1$ now gives a contradiction $N=N+1$. 

\vspace{0.07in} 

{\bf VI.}
The most general case is when, for a given eigenvalue $\lambda_i$, there's at most one nonsemisimple Jordan block and several $1\times 1$ blocks. This case requires only a minor modification of the above argument, with only the the column matrix $R$ on the right of \eqref{eq_TG} changed. It is exactly the  difference between steps {\bf III} and {\bf IV}, see earlier, with the matrix $T$ unchanged and the unique solution given by replacing $N_i$ by $M_i$ in \eqref{eq_gamma_3}, where $M_i-N_i$ is the number of semisimple blocks of eigenvalue $\lambda_i$.  

\vspace{0.07in} 

This concludes our proof that a system $(p,\iota,h)$ as above with $\dim \ker (h)=1$ does not exist. 

\begin{remark}
    The matrix $T$ in \eqref{eq_T} is a variation on the \emph{confluent Vandermonde matrix}. These matrices and their determinants appear in the theory of random matrices \cite{Meh67}, in approximation theory and Diophantine geometry, see \cite{MatSep2018} and references therein.  
\end{remark}

\begin{remark} A alternative short proof of the impossibility $m=1$ is as follows. 
The sequence $p\, h^n \iota$ is a fixed linear combination of entries of matrix $h^n$, i.e. linear combination of sequences of the form
$\lambda_i^n, n\lambda_i^n, n(n-1)\lambda^n$ etc. (finitely many sequences taken from the formula \eqref{eq_h_power} for $h^n$ and its version for $h$ with multiple Jordan blocks and various eigenvalues $\lambda_i$). It is a classical
result that these sequences (and even their tails when we restrict to 
 $n>K$ for some fixed number $K$) are linearly independent. 
This is easy to prove directly: if there is a linear dependence, there is a similar linear dependence of generating functions. If one 
considers tails for $n>K$, then there is a linear dependence
of generating functions modulo polynomials of degree at most $K$.
Now the result follows since the generating functions are
$1/(1-\lambda_i T)^s$ for various $\lambda_i\not=0$ and $s>0$.

 On the other hand, we have that the sequence $p\, h^n \iota$ equals the sequence $\tr_V(h^{n-1})$ for $n\ge 2$.
It follows that coefficients in linear combination for $p\, h^n \iota$ are exactly the same as coefficients for the sequence $\tr_V(h^{n-1})$; thus the same formula
applies even for $n=1$ which gives a contradiction due to the presence of additional $1$ in $\tr_V(h^0)=\tr_V(\I)$ corresponding to the eigenvalue $0$ subspace. 

This obstruction is unique to $m=1$. 
If $m=2$ (and, more generally, $m>2$) there could be a  $2\times 2$ Jordan
cell with eigenvalue $0$ in the decomposition of $h$. Then one of the matrix entries of $h^n$ is the sequence
$(0,1,0,0,0,\dots)$. 
Using this sequence one can adjust an incorrect value
of $\alpha_1$, avoiding the contradiction. 
\end{remark}

We have included both proofs of impossibility of the $m=1$ case in this paper. The first proof is long but exhibits an interesting connection to the confluent Vandermonde determinant and Diophantine approximations~\cite{Meh67,MatSep2018}. 

\vspace{0.07in} 

Putting the cases $m\ge 2$, $m=1$ and $m=0$ together
gives a classification of pseudocharacters for two-dimensional topological theories over algebraically closed fields of characteristic $0$.

\begin{theorem} \label{thm_classif} Suppose that $\alpha$ is a pseudo-TQFT (a pseudocharacter) for the category $\Cob_2$ of oriented two-dimensional cobordisms taking values in an algebraically closed field $\kk$ of characteristic $0$. Then
the generating function $Z_{\alpha}(T)$ is rational and has the form  \eqref{eq_Z_type1} or \eqref{eq_Z_type2} as in Proposition~\ref{prop_gen_func}. In particular, any such pseudocharacter embeds into a two-dimensional TQFT for $\Cob_2$ given by some commutative Frobenius algebra $(B,\varepsilon)$ over $\kk$. 
\end{theorem} 

This result can be viewed as the first step in studying pseudocharacters beyond dimension one. While the one-dimensional case, needed  in number theory, has $G$-labelled defects placed on a one-manifold, the above theorem is for two-manifolds without defects. Possible extensions of that result to two-manifolds with defects are worth investigating. For instance, placing labelled zero-dimensional defects on a surface corresponds to coupling category $\Cob_2$ to a commutative algebra, see~\cite[Section 8]{KKO}. A classification of pseudocharacters for such decorated $\Cob_2$ categories may extend the work of Buchstaber and Rees~\cite{BR2004} on pseudocharacters for commutative rings and its generalization to the super case~\cite{KV20}.

%%%   Subjects:  
%%% 57K16 Finite-type and quantum invariants, topological quantum field theories (TQFT), 
%%% 18M05 Monoidal categories, symmetric monoidal categories, 
%%% 18M30 String diagrams and graphical calculi, 
%%% 15A15 Determinants, permanents, traces, other special matrix functions. 
%%% QA: Quantum Algebra, RT: Representation Theory,  MP: Mathematical Physics,  CT: Category Theory. 

%%%%%%%%%%%%%%%%%%%%%
%%
%%   REFERENCES 
%%
%%%%%%%%%%%%%%%%%%%%

%\addcontentsline{toc}{section}{References}
%\def\refname{}

\bibliographystyle{amsalpha} 
%\bibliographystyle{amsplain}

%\begin{thebibliography}{A}
%\begin{thebibliography}{alpha}
%\raggedright

\bibliography{top-automata}

\end{document}

%% file: figure-D0.tex
\begin{figure}
    \centering
\begin{tikzpicture}[scale=0.6]
\begin{scope}[shift={(0,0)}]
%\draw[thin,yellow] (0,0) grid (4,4);
\draw[thick,->] (1,2) arc (90:-270:1); 
\node at (1.95,2.25) {$\SS_X$};

\node at (-0.5,1) {$X$};

\node at (4.125,1.35) {$\alpha$};
\draw[thick,->] (3.25,1) -- (5.25,1);

\node at (7.75,1) {$\rk(\mcF(X))$};

\end{scope}

\end{tikzpicture}
    \caption{Endomorphism $\SS_X$ of $\one$ (the circle carries the identity endomorphism $\id_X$) and its $\alpha$-evaluation (the rank of the projective module $\mcF(X)$).}
    \label{figure-D0}
\end{figure}

%% file: figure_gauge001.tex
\begin{tikzpicture}[scale=0.6]
%\draw[thin,yellow] (0,0) grid (4,3);

\node at (0,2) {Gauge-equivalence classes of};
\node at (0,1.25) {connections on $\R^n$-bundles over $M$ };
\draw[thick,->] (6,1.5) -- (8,1.5);
\node at (13,2) {$\R$-valued pseudocharacters};
\node at (13,1.25) {of $\mcB(\mcC_M)$ of degree $n$};

\end{tikzpicture}